\documentclass[10pt]{amsart}
\usepackage{palatino, amssymb, amsfonts, latexsym, mathrsfs,comment}
\usepackage{amssymb, amsmath, amscd, color}
\usepackage{verbatim}
\input xy
\xyoption{all}

\newdir^{ (}{{}*!/-5pt/@^{(}}

\newtheorem{theorem}{Theorem}[section]
\newtheorem{lemma}[theorem]{Lemma}
\newtheorem{prop}[theorem]{Proposition}
\newtheorem{cor}[theorem]{Corollary}
\theoremstyle{definition}
\newtheorem{definition}[theorem]{Definition}
\newtheorem{example}[theorem]{Example}

\newtheorem*{convention}{Convention}
\newtheorem{question}[theorem]{Question}

\theoremstyle{remark}
\newtheorem{remark}[theorem]{Remark}

\newcommand{\on}{\operatorname}

\newcommand{\C}{\mathbf C}

\newcommand{\Lw}{\Lambda_\mathbf w}
\newcommand{\Qv}{\alpha_{\mathbf v}}

\newcommand{\fa}{\mathfrak a}

\newcommand{\Gr}{\mathsf{Gr}}
\newcommand{\llsb}{[\![}
\newcommand{\rrsb}{]\!]}
\newcommand{\llb}{(\!(}
\newcommand{\rrb}{)\!)}
\newcommand{\Glm}{\Gr_{\mu}^{\bar{\lambda}}}
\newcommand{\Glmr}{\Gr^{\vec{\lambda}}_\mu}
\newcommand{\gf}{\mathfrak g}

\newcommand{\Gl}{\Gr^{\lambda}}
\newcommand{\Glr}{\Gr^{\vec{\lambda}}_0}
\newcommand{\Spr}{\mathsf{Spr}}

\numberwithin{equation}{section}

\begin{document}

\title{Springer theory for symplectic Galois groups}

\author{Kevin McGerty}
\address[McGerty]{Mathematical Institute\\University of Oxford\\Oxford OX1 3LB, UK}
\email[McGerty]{mcgerty@maths.ox.ac.uk}
\address[McGerty]{Department of Mathematics\\University of Illinois at Urbana-Champaign\\Urbana, IL 61801 USA}
\author{Thomas Nevins}
\address[Nevins]{Department of Mathematics\\University of Illinois at Urbana-Champaign\\Urbana, IL 61801 USA}
\email[Nevins]{nevins@illinois.edu}

\date{\today}

\begin{abstract}
A classical and beautiful story in geometric representation theory is the construction by Springer of an action of the Weyl group on the cohomology of the fibres of the Springer resolution of the nilpotent cone.   We establish a natural extension of Springer's theory to arbitrary symplectic resolutions of conical symplectic singularities.  We analyse features of the action in the case of affine quiver varieties, constructing Weyl group actions on the cohomology of $ADE$ quiver varieties, and also consider ``symplectically dual'' examples arising from slices in the affine Grassmannian. Along the way, we document some basic features of the symplectic geometry of quiver varieties.
\end{abstract}

\maketitle

\section{Introduction}

\vspace{1em}

A complex {\em symplectic variety} is a (quasiprojective) variety $X$ with at worst rational, Gorenstein singularities and with a symplectic form $\omega$, i.e. a closed, nondegenerate algebraic $2$-form, on its smooth locus.  A resolution of singularities $Y\rightarrow X$ is a {\em symplectic resolution} if $\omega$ extends to a symplectic form on all of $Y$.

A celebrated example is the cotangent bundle to a flag variety of a reductive group, which, as was first observed by a number of people including Springer and Steinberg, is a symplectic resolution of the nilpotent cone of the corresponding Lie algebra. 
In recent years,
much effort has been focused on discovering to what extent one can generalise known representation-theoretic phenomena from Springer resolutions to all symplectic resolutions. 

It is a beautiful observation of Markman \cite{M} in the projective case---extended to the conical affine case by Namikawa \cite{Na2}---that if $\pi\colon Y\to X$ is a symplectic resolution, then the Poisson deformation base of $Y$ is a Galois covering of the corresponding deformation base of $X$.
The corresponding ``symplectic Galois group,'' in the case of the cotangent bundle of the flag variety, is precisely the Weyl group, and in general it is the product of reflection groups attached to the geometry of slices to the resolution in codimension two. 

\vspace{.7em}

In the present paper, we show that the cohomologies of the fibres of a symplectic resolution carry representations of the symplectic Galois group, generalising the seminal work of Springer \cite{Sp}. We also show that these representations can be constructed in a number of ways, mirroring the various constructions in the classical case: in particular one can use both small resolutions and a version of the nearby cycles construction over a higher dimensional base.  Along the way, we establish some fundamental but not well-documented features of the symplectic geometry of Nakajima quiver varieties, and compute their symplectic Galois groups in the finite-type case.  

\vspace{.7em}

A more precise description of our results is the following.  Suppose that $X$ is a conical symplectic variety (Definition \ref{def:conical}) with symplectic resolution $Y\xrightarrow{\pi} X$.  Let $\mathcal{Y} \rightarrow B_Y$, $ \mathcal{X}\rightarrow B_X$ be the corresponding versal Poisson deformations, with morphism $\mathcal{Y}\xrightarrow{\tilde{\pi}} \mathcal{X}$, as constructed in \cite{M, Na2} (and reviewed in Section \ref{sec:defthy}), with symplectic Galois group $W$.  Fixing a coefficient field $k$,
we define constructible complexes
\begin{displaymath}
\mathsf{HC} := \widetilde{\pi}_!k_{\mathcal{Y}}[\operatorname{dim}(\mathcal{Y})] \hspace{2em} \text{and} \hspace{2em}
 \mathsf{Spr} := \pi_!k_{Y}[\operatorname{dim}(Y)],
 \end{displaymath}
  the {\em symplectic Harish-Chandra sheaf}
  and {\em symplectic Springer sheaf}.
  Let $\operatorname{Perv}_{\operatorname{sympl}}(X)$ denote the category of perverse sheaves on $X$ (with coefficients in $k$) which are smooth along the stratification by symplectic leaves.  We also define, in Section \ref{nearby cycles}, a nearby cycles sheaf $\mathcal{P}$ for the family $\mathcal X \to B_X$
   (again, notation as in Section \ref{sec:defthy}).

\begin{theorem}[Theorem \ref{Springer properties}, Corollary \ref{cor:action}, Theorem \ref{nearby cycles2}, Proposition \ref{monodromy}]
\mbox{}
\begin{enumerate}
\item The complexes $\mathsf{HC}$ and $\mathsf{Spr}$ are perverse sheaves, which are semisimple if $\on{char}(k)=0$. Moreover $\mathsf{HC}$ is the intersection cohomology extension of the local system on $\mathcal X^{\text{reg}}$ given by $k[W]$, the regular representation of the symplectic Galois group $W$.
\item We have $i_X^*\mathsf{HC} \cong \mathsf{Spr}$, where $i_X\colon X  \hookrightarrow \mathcal X$ is the inclusion. 
\item We have a natural algebra isomorphism
\begin{displaymath}
k[W]\xrightarrow{a} \operatorname{End}(\mathsf{HC}).
\end{displaymath}
\item We get an adjoint pair of functors:
\begin{displaymath}
(-\otimes_{k[W]} \mathsf{Spr}): k[W]-\operatorname{mod} \leftrightarrows \operatorname{Perv}_{\operatorname{sympl}}(X): \operatorname{Hom}(\mathsf{Spr}, -).
\end{displaymath}
\item The group $W$ acts faithfully on the sheaf $\mathsf{Spr}$, yielding, for each $x\in X$, an action of $W$ on $H^*\big(\pi^{-1}(x)\big)$.
\item  We have $\mathcal P \cong \pi_!(k_{Y})$ and the action of the braid group factors through the symplectic Galois group $W$.
\item The monodromy action of $W$ on $\pi_!(k_Y)$ induced by the isomorphism $\mathcal P \cong \pi_!(k_Y)$ coincides with the action of $W$ induced by the intermediate extension.
\end{enumerate}
\end{theorem}
\noindent
The theorem establishes the basic viability of Springer theory in a symplectic setting.  In Section \ref{sec:ex-conj} we discuss some examples and a question about multiplicities of Springer sheaves; it would be very interesting to have more data concerning the question discussed there.

Turning to Nakajima's quiver varieties, we review several features of their geometry, most likely known to experts, but not all of which appear yet to be documented in the literature.  Fixing a quiver $Q$ as in Section \ref{sec:set-up}, let $Y(\mathbf v,\mathbf w)$, respectively $X(\mathbf v,\mathbf w)$, be the smooth symplectic quiver variety and affine quiver variety associated to the dimension vectors $\mathbf v$ and $\mathbf w$.  
\begin{theorem}[\cite{BS}, \ref{smooth is resolution}, and \ref{prop:res-exists}]
\mbox{}
\begin{enumerate}
\item $X(\mathbf v,\mathbf w)$ is a conical affine symplectic variety.
\item $Y(\mathbf v,\mathbf w)$ is a symplectic resolution of some conical affine symplectic variety.
\item If $Q$ is a Dynkin quiver of finite type, then for each $\mathbf v$ and $\mathbf w$ there exists $\mathbf{v}'$ so that
$X(\mathbf v',\mathbf w)\cong X(\mathbf v,\mathbf w)$ and  $\pi_{\mathbf v'}\colon Y(\mathbf v',\mathbf w)\to X(\mathbf v',\mathbf w)$ is a symplectic resolution.
\end{enumerate}
\end{theorem}
\noindent
We caution the reader that in general  $Y(\mathbf v,\mathbf w)$ need {\em not} be a resolution of $X(\mathbf v,\mathbf w)$ (cf. Example \ref{quiver example}). In \cite{BS} Bellamy and Schedler, as well as establishing part $(1)$ of the above theorem, also determine precisely when the varieties $X(\mathbf v,\mathbf w)$ have symplectic resolutions. 

Furthermore, we calculate the symplectic Galois group of quiver varieties for finite-type Dynkin quivers.  More precisely, associated to the finite-type Dynkin quiver $Q$ and dimension vectors $\mathbf v$ and $\mathbf w$ is a parabolic subroot system (denoted $\Phi_{\mu}$ in Section \ref{sec:quiver-Weyl}) with associated Weyl group $W_{\mu}$.  We show:
\begin{theorem}[Theorem \ref{QuiverWeylgroup}]\label{quiverweylgroup}
If $Q$ is a Dynkin quiver of finite type, the symplectic Galois group of the quiver variety $X(\mathbf{v},\mathbf{w})$ is $W_\mu$.
\end{theorem}
\noindent
Combining Theorem \ref{quiverweylgroup} with our general theory, we obtain Weyl group actions on the cohomology of ADE quiver varieties. Weyl group actions have previously been obtained by Lusztig \cite{Lu}, Nakajima \cite{N03} and Maffei \cite{Maffei}.

Outside of finite type, as we show in Example \ref{affine example}, the naive generalization of Theorem \ref{quiverweylgroup} fails. It is not clear to the authors what the correct general statement should be. It is plausible that one needs to take into account the existence of actions by larger algebras than the Kac-Moody algebras on topological invariants of quiver varieties.

We should note that much has already been written about categorical actions of braid groups on (derived categories of) quiver varieties and their quantizations.  It is fairly clear that, when comparable, our actions decategorify those (cf. Lemma \ref{Springer-via-BM} and the discussion around it).  Even if one views our actions, at a lower categorical level, as a poor relation of categorical actions, our approach does work uniformly, independent of coefficient field.  It is also quite concrete and topological; one important feature of Springer theory is the interplay amongst various approaches.

\begin{convention}
Throughout the paper, all varieties and schemes are over the field $\mathbb{C}$ of complex numbers, unless otherwise stated.  We write $k$ to denote a coefficient field for sheaves, cohomology, etc.
\end{convention}

We are grateful to many people for helpful conversations and comments, in particular Gwyn Bellamy, Tom Braden, Mark Goresky, Dmitry Kaledin, and Joel Kamnitzer.  Both authors are grateful to MSRI  for excellent working conditions during the initial work on this project.
The first author was supported by a  Royal Society University Research Fellowship and the EPSRC Programme Grant on motivic invariants and categorification, and by a Fisher Visiting Professorship at the University of Illinois at Urbana-Champaign.  The second author was supported by NSF grants DMS-1159468, DMS-1502125, and DMS-1802094 and NSA grant H98230-12-1-0216, and by an All Souls Visiting Fellowship.  Both authors were supported by MSRI.

\section{Symplectic resolutions}
\label{Symplectic resolutions}
\subsection{Symplectic Varieties and Poisson Structures}

Symplectic varieties were introduced by Beauville \cite{Be}.

A {\em symplectic variety} $X$ is a normal variety which admits an (algebraic) symplectic form $\omega$ on its smooth locus $X^\text{reg}$ such that, for some (or equivalently, every) resolution of singularities $\pi \colon Y \to X$, the pull-back of $\omega$ extends to a regular (though not necessarily non-degenerate) $2$-form on all of $Y$.  In the case where the pull-back of $\omega$ does extend to a symplectic form on all of $Y$, we say that $\pi$ is a {\em symplectic resolution}. It is not necessarily the case that a symplectic variety admits a symplectic resolution (see the work of Fu \cite{Fu} on nilpotent orbit closures for examples). 

We restrict our attention to affine symplectic varieties and their symplectic resolutions. These have been studied by a number of people.  We will rely heavily on Kaledin's lovely paper \cite{Ka},  in which he provides a basic structure theory for symplectic varieties which we now review. We begin by noting that, since $X$ is normal, the Poisson bracket on functions on $X^{\text{reg}}$ given by $\omega$ extends to a Poisson bracket on all of $X$. Thus, an irreducible affine symplectic variety is an integral Poisson scheme.

\begin{definition}
We say that a Poisson variety $Z$ is \textit{generically nondegenerate} if the bivector giving the Poisson bracket yields a generically nondegenerate alternating two-form. A Poisson variety is said to be  \textit{holonomic} if the restriction of the Poisson bracket to any integral Poisson subscheme is generically nondegenerate.
\end{definition}
\noindent

Recall that the {\em singularity stratification} of a reduced scheme $X$ of finite type is the stratification whose closed subsets $X = X_0 \supset X_1\supset X_2\supset \dots \supset X_k$ have the property that $X_k^\circ := X_k\smallsetminus X_{k+1}$ equals $X_k^{\operatorname{reg}}$, the smooth locus of $X_k$.  This stratification is canonical and thus is preserved by all automorphisms of $X$ and vector fields on $X$.

\begin{theorem}
\cite{Ka} Let $X$ be a complex symplectic variety.  The locally closed strata of the singularity stratification of $X$ are locally closed smooth symplectic subvarieties. The closures $X_i$ of the $X_i^0$ are irreducible Poisson subschemes and every integral Poisson subscheme of $X$ is the closure of a stratum. Moreover the normalization of any Poisson subscheme $Y \subset X$ is also a symplectic variety. 
\label{Kaledin strata}
\end{theorem}

Note that, in particular, a symplectic variety is a holonomic Poisson variety. The converse, however, is false (already in dimension 2, for example for simple elliptic singularities).

\subsection{Deformation Theory and the Symplectic Galois Group}\label{sec:defthy}
\begin{definition}\label{def:conical}
A {\em conical symplectic variety} is an affine symplectic variety $X$ equipped with a $\mathbb{G}_m$-action with nonnegative weights on $\mathbb C[X]$ and $\mathbb{C}[X]^{\mathbb{G}_m}  =\mathbb{C}$, and for which the symplectic form transforms with weight $d >0$.
\end{definition}

While standard deformation theory for non-projective or singular varieties is generally rather intractable, it turns out there is a sensible theory of \textit{Poisson deformations} for pairs $(X,\{-,-\})$ consisting of a conical affine symplectic variety $X$ with its Poisson bracket.  This deformation theory  has been studied by Namikawa \cite{Na1} and Ginzburg-Kaledin \cite{GK}; here we will mostly follow \cite{Na1}. In that paper, Namikawa defines a Poisson deformation functor $PD_X$ from the category of local Artinian rings with residue field $\mathbb C$ to sets and analyses it using the Poisson cohomology $HP^*(X)$. 
In the case where $X$ has only terminal singularities, or is affine with $H^1(X^{\text{an}}, \mathbb C) = 0$, it is known \cite[Corollary 2.5]{Na1} that the Poisson deformation functor is pro-representable. 

In the case where $X$ is a symplectic variety and $\pi\colon Y\to X$ is a symplectic resolution, we thus have Poisson deformation functors for both $X$ and $Y$.
It can be shown that if $X$ is a conical symplectic variety, then the $\mathbb G_m$-action lifts to the resolution $Y$ and to the pro-representing hulls $R_X$ and $R_Y$. For such a conical symplectic variety $X$, the hull $R_Y$ associated to the resolution $Y$ is a formal power series ring 
\[
\mathbb C[[y_1,\ldots,y_d]]\cong\text{Sym}(H^2(Y,\mathbb C)^*)^{\wedge}
\]
where $\mathbb G_m$ acts with weight $d$ on each $y_i$. Now $R_X$ is a subring of $R_Y$, hence the $\mathbb G_m$-action again has positive weights on $R_X$. These actions allow one to algebraicize the universal formal Poisson deformations of $X$ and $Y$.  Indeed, let $R^{\mathbb{G}_m-\operatorname{fin}}$ denote the subspace of $\mathbb{G}_m$-finite vectors in a $\mathbb{G}_m$-representation $R$.  Writing $B_X = \operatorname{Spec}(R_X^{\mathbb{G}_m-\operatorname{fin}})$ and similarly for $Y$, one obtains compatible $\mathbb G_m$-equivariant Poisson deformations of $X$ and $Y$:

 \begin{equation}\label{main diagram}
\xymatrix{
Y \ar@{^{ (}->}[r]^{i_Y} \ar[d]^\pi &\mathcal Y \ar[r]^{\theta_Y} \ar[d]^{\tilde{\pi}} & B_Y \ar[d]^{q}  \\
X  \ar@{^{ (}->}[r]^{i_X}  & \mathcal X \ar[r]^{\theta_X} & B_X.
}
\end{equation}
 \noindent
 \begin{displaymath} 
\text{Here} \hspace{1em} B_Y \cong  H^2(Y,\mathbb C)\cong \mathbb{A}^d \cong B_X, \hspace{1em} \text{where } d = \dim(H^2(Y,\mathbb C)).
 \end{displaymath} 
  The identification of the base of the deformation of $H^2(Y,\mathbb C)$ arises via a symplectic version of a period map: in general this identifies the base of the formal deformation with the formal neighbourhood of the class of the symplectic form $\omega_Y$ in $H^2(Y,\mathbb C)$. In the conic case of course, $[\omega_Y]=0$. Note also that the affinization of the upper row is the pull-back of $Y$ along the covering of $B_Y$ by $H^2$.

\subsection{The Universal Conic Deformation}
\label{sec:Conic universality}
We wish to understand to what extent the families $\mathcal Y$ and $\mathcal X$ are universal\footnote{Everything here is presumably well known to the experts, but we failed to find a suitable reference, so for the convenience of the reader we outline the details here.}: to do this, we consider families $f\colon \mathcal X_\fa \to \fa$ of symplectic resolutions over a vector space $\mathfrak a$, such that the $d$-th power of the scaling action on the vector space $\fa$ lifts to an action on the symplectic family $\mathfrak X_\fa$ inducing the action $\rho$ on the central fibre $X$, with moreover $t.\tilde{\omega} = t^d\tilde{\omega}$, where $\tilde{\omega}$ is the symplectic form of $\mathfrak X_\fa$ over $\fa$.  We shall call such a family a weight $d$ \textit{conic deformation}.

\begin{prop}
\label{conic universality}
The family $\mathcal Y\to H^2(Y,\mathbb C)$ is universal among weight $d$ conic deformations, that is, any such family $\mathcal Y_\fa \to \fa$ is obtained by pull-back from a linear map $\alpha\colon \fa \to H^2(Y,\mathbb C)$. 
\end{prop}
\begin{proof}
We sketch a proof.  First we note that by Kaledin and Verbitsky \cite{KV}, there is an universal \textit{formal} deformation of $(Y,\omega)$, which we shall denote by $(\hat{\mathcal Y},\hat{\omega})$ over a formal base $S$. (Note that any symplectic resolution is admissible in the terminology of $\S 2$ of \cite{KV}.) This formal base is the completion at $0=[\omega_Y]$ of $H^2(Y,\mathbb C)$. Now the universality of this deformation ensures that the $\mathbb G_m$ action extends to an action of the formal group $\hat{\mathbb G}_m$ on $\hat{\mathcal Y}$. 

Now as discussed above, the formal action allows us to spread out $S$ to $H^2(Y,\mathbb C)$ in such a way that the formal $\mathbb G_m$-action extends uniquely to become a weight $d$ action. Given this, if we have any weight $d$ conic deformation $\mathcal Y_\fa\to \fa$, then by universality its formal completion at $0$ is pulled back from the universal deformation $(\hat{\mathcal Y},\hat{\omega})$ by a map $\tilde{f}$ say. The $\mathbb G_m$-actions then allow us to spread this out to a map $f\colon \fa \to H^2(Y,\mathbb C)$, which by homogeneity (everything in sight has weight $d$ for the $\mathbb G_m$-action) must be linear.
\end{proof}

\begin{remark}
The Poisson deformation family $\mathcal X \to B_X$ is also shown by Namikawa \cite{Na2} to be formally universal. The $\mathbb G_m$-action then allows this to be algebraicised to be universally conic.
\end{remark}

 \subsection{The Symplectic Galois Group}

It is shown in \cite{Na2} that $q$ is a Galois covering.
\begin{definition}
We call the covering group $W$ of $q$ the {\em Galois group} of our symplectic resolution, or simply the {\em symplectic Galois group}. 
\end{definition}

Namikawa's work shows that $W$ can be described quite explicitly in terms of the exceptional divisors of the resolution $\pi$: Let $\Sigma$ denote the singular locus of $X$ and $\Sigma_0$ denote the union of the strata of $X$ of codimension $4$ or higher, and set $U = X \backslash \Sigma_0$. Let $\mathcal B$ denote the set of connected components of $\Sigma\backslash \Sigma_0$, the singular locus of $U$. For each $B$ in $\mathcal B$, a transverse slice $S_B$ through a point $b \in B$ yields a two-dimensional symplectic singularity---that is, a rational surface singularity. Let $W'_B$ denote the Weyl group of type $ADE$ attached to this singularity. The monodromy, as $b \in B$ varies, induces a diagram automorphism of $W'_B$ and we set $W_B\subset W'_B$ to be its fixed-point subgroup. The main theorem of \cite{Na2} then shows that the symplectic Galois group is 
\[
W = \prod_{B \in \mathcal B} W_B,
\]
and so $B_X = H^2(Y,\mathbb C)/W$.

Note that in Diagram \eqref{main diagram} each square is commutative, but the right-hand square is \textit{not} Cartesian, since for example the fibre product $\mathcal X' := \mathcal X\times_{B_X}B_Y$ is affine. Following Namikawa \cite{N3}, we let $\Pi \colon \mathcal Y \to \mathcal X'$ be the natural map and, for $t \in H^2(Y,\mathbb C)$, let $\Pi_t\colon \mathcal Y_t \to \mathcal X'_t$ be the restriction of $\Pi$ to the preimage $\mathcal Y_t = \psi_Y^{-1}(t)$.

We now recall a result of Namikawa. Let $\mathcal D\subset H^2(Y,\mathbb C)$ be the discriminant locus: the closure of the locus where $\Pi_t$ is not an isomorphism. For a $\mathbb{Q}$-vector space $H$ write $H_{\mathbb C}= \mathbb C\otimes_{\mathbb Q} H$.
\begin{theorem}[\cite{N3}]
\label{Namikawa discriminant}
There are finitely many linear subspaces $\{H_i\}_{i \in I}$ of $H^2(Y,\mathbb Q)$ of codimension $1$ such that $\mathcal D = \bigcup_{i \in I} (H_i)_{\mathbb C}$. Morever the open chambers in $H^2(Y, \mathbb{R})$ complementary to the arrangement of real hyperplanes $\{(H_i)_{\mathbb R}\}_{i \in I}$ are given by $\{w(\text{Amp}(\pi_k))\}$ where $w \in W$ and $\pi_k$ are the projective  crepant  resolutions of $X$, and $\text{Amp}(\pi_k)$ denotes the ample cone of $\pi_k$.
\end{theorem}

 Thus, in particular, the discriminant locus of the map $\Pi$ is an algebraic hypersurface: more precisely, a finite union of hyperplanes.

\subsection{(Semi)Smallness and Versal Deformations}
\begin{definition}
A morphism of varieties $f\colon Y \to X$ is said to be \textit{semismall} if the fibre product $Z =Y\times_X Y$ has dimension at most $d=\dim(Y)$. If moreover any irreducible component of $Z$ which has dimension $d$ maps to a dense subset of $X$ then $f$ is said to be \textit{small}.  
\end{definition}
Kaledin shows that any symplectic resolution $Y\rightarrow X$ is semismall.  This fact can be refined as follows.
\begin{prop}
\label{smallness}
Let $\pi \colon Y\to X$ be a conical symplectic resolution, and let $\widetilde{\pi}: \mathcal{Y}\rightarrow \mathcal{X}$ be the morphism from the total space of the universal Poisson deformation of $Y$ to the corresponding deformation of $X$ itself.  Then  $\widetilde{\pi}$ is small.
\end{prop}
\begin{proof}
Let $\mathcal X_{\operatorname{sing}}$ denote the locus in $\mathcal X$ which is the union of the deformations $\mathcal X_t$ ($t \in B_X$) such that $\mathcal X_t$ is singular. By the main result of \cite[Main Theorem]{N3} this is the preimage $\psi_X^{-1}(\bar{\mathcal D})$ of a proper closed subset $\bar{\mathcal D}$ of $B_X$ the base of the universal Poisson deformation, and moreover, the map $\tilde{\pi}$ is finite over $\mathcal X\smallsetminus \mathcal X_{\operatorname{sing}}$. To prove the Proposition it thus suffices to check that $\dim(\mathcal Y\times_{\mathcal X_{\operatorname{sing}}} \mathcal Y)< \dim(\mathcal Y)$.

To see this, for each $p\in H^2(Y)$, let $\tilde{\pi}_p\colon \mathcal Y_p \to \mathcal X_{q(p)}$ denote the restriction of $\tilde{\pi}$ to $\mathcal Y_p =\psi_Y^{-1}(p)$. Since $\tilde{\pi}_p$ is then a projective birational morphism from the smooth variety $\mathcal Y_p$ to the symplectic variety $\mathcal X_p$ it follows immediately from Lemma $2.11$ of \cite{Ka} that $\tilde{\pi}_p$ is semismall so that $\dim(\mathcal Y_p\times_{\mathcal X_{q(p)}}\times \mathcal Y_p) = \dim(\mathcal Y_p)=\dim(Y)$.  Hence by considering the natural map $\mathcal Y \times_\mathcal X \mathcal Y \to B_X$ we see that
\begin{displaymath}
\operatorname{dim}(\mathcal{Y}\times_{\mathcal{X}_{\operatorname{sing}}}\mathcal{Y}) \leq \on{dim}(Y) + \on{dim}\big(\psi_X(\mathcal{X}_{\on{sing}})\big) < \on{dim}(\mathcal{Y}),
\end{displaymath}
as required.
\end{proof}

\begin{remark}
\label{smallness of conic deformations}
Note that the above argument works for any conic deformation $\mathcal Y_\fa\to \fa$ of a conic symplectic resolution $\pi\colon Y \to X$ whose generic deformation is affine: If $\mathcal X_\fa$ is the affinization of $\mathcal Y_\fa$ then the morphism $\tilde{\pi}_\fa \colon \mathcal Y_\fa \to \mathcal X_\fa$ is small. This follows readily from the above result and Proposition \ref{conic universality}: if the generic deformation of the family $\mathcal Y_\fa$ is affine, then the linear map $\alpha\colon \fa \to H^2(Y,\mathbb C)$ which classifies $\mathcal Y_\fa$ has its image $\text{im}(\alpha)$ not contained in the discriminant locus $\mathcal D\subseteq H^2(Y,\mathbb C)$. The argument of the previous Proposition then applies unchanged to $\tilde{\pi}_\fa$.
\end{remark}

\section{Symplectic Geometry of Quiver Varieties}
We investigate  symplectic features of Nakajima's quiver varieties.  The content of this section is largely well known to experts, but we document the details where we could not find suitable references. 

\subsection{Basic Set-Up}\label{sec:set-up}
We briefly recall the setting, using \cite{N98} as our principal reference.
Let $Q= (I,\Omega)$ be a quiver with vertex set $I$ and edge set $\Omega$, such that the edges $h \in \Omega$ are directed with starting vertex $s(h)$ and terminating vertex $t(h)$. Let $\bar{Q}$ denote the doubled quiver, \textit{i.e.} the quiver with vertex set $I$ but edge set $H = \Omega\sqcup \bar{\Omega}$ where $\bar{\Omega}$ is in bijection with $\Omega$ via a map $h \mapsto \bar{h}$ so that $s(\bar{h}) = t(h)$, and $t(\bar{h}) = s(h)$. For convenience we treat $\overline{(\cdot)}$ as an involution on $H$ by setting $\overline{(\overline{h})} = h$.

A quiver variety is defined by a pair $\mathbf v,\mathbf w \in \mathbb Z_{\geq 0}^I$. Given $\mathbf v$, $\mathbf w$, choose $I$-graded vector spaces $V$ and $W$ with $\dim(V_i)= \mathbf v_i$, $\dim(W_i) = \mathbf w_i$ for each $i \in I$, and let 
\[
\mathbf M(\mathbf v,\mathbf w) = \bigoplus_{h \in H}\text{Hom}(V_{s(h)},V_{t(h)}) \oplus \bigoplus_{i \in I}\text{Hom}(W_i,V_i) \oplus \bigoplus_{i \in I}\text{Hom}(V_i,,W_i).
\]
We will write an element of this space as a triple $(x,p,q)$. Picking a function $\varepsilon\colon H \to \{\pm 1\}$ such that $\varepsilon(h) + \varepsilon(\bar{h})=0$, the formula
\[
\omega\big((x,p,q),(x',p',q')\big) = \sum_{h\in H}\text{tr}\big(\varepsilon(h)x_hx_{\overline{h}}'\big) + \sum_{i\in I}\text{tr}\big(p_i'q_i+p_iq_i'\big)
\]
defines a symplectic form on $M(\mathbf v,\mathbf w)$ for which the action of the group $G_\mathbf v = \prod_{i \in I} \text{GL}(V_i)$ is Hamiltonian, with moment map:
\[
\mu(x,p,q) = \left(\sum_{h:s(h)=i} \varepsilon(h)x_{\bar{h}}x_h + pq\right)_{i \in I} \in \mathfrak{g}_\mathbf v^* = \bigoplus_{i \in I}\mathfrak{gl}_{V_i}
\]
(where we identify $\mathfrak{gl}_V$ with its dual using the trace pairings). We may then consider two kinds of quotients of $M(\mathbf v,\mathbf w)$ by the action of $G_\mathbf v$. The first is the affine quotient, which we will denote by $X(\mathbf v,\mathbf w)$. That is,
\[
X(\mathbf v,\mathbf w) = \text{Spec}(\mathbb C[\mu^{-1}(0)]^{G_\mathbf v}).
\]
The points of $X(\mathbf v,\mathbf w)$ thus correspond to the closed $G_\mathbf v$-orbits in $\mu^{-1}(0)$. 
Notice that the Poisson bracket associated to the symplectic form $\omega$ is clearly preserved by the action of $G_\mathbf v$, so that it induces a Poisson structure on $\mathbb C[\mu^{-1}(0)]^{G_\mathbf v}$, and hence $X(\mathbf v,\mathbf w)$ is a Poisson variety.

The second quotient is the GIT quotient determined by the linearisation of the trivial line bundle given by the determinant character $\chi$ (more generally, one could take any sufficiently general character of $G_{\mathbf{v}}$ that is trivial on the diagonal $\mathbb{G}_m$): 
\[
Y(\mathbf v,\mathbf w) = \text{Proj}(\bigoplus_{n \geq 0} \mathbb C[\mu^{-1}(0)]^{\chi^n}).
\]
In \cite{N98} it is shown that $Y(\mathbf v,\mathbf w)$ coincides with the quotient of $\mu^{-1}(0)^s$ by the action of $G_\mathbf v$, where $\mu^{-1}(0)^s$ is the set of stable points of $\mu^{-1}(0)$, and $G_\mathbf v$ acts freely there, so the quotient is again an algebraic symplectic manifold. The description of $Y(\mathbf v,\mathbf w)$ immediately shows that there is a projective morphism $\pi \colon Y(\mathbf v,\mathbf w) \to X(\mathbf v,\mathbf w)$. It is easy to check also that the variety $Y(\mathbf v,\mathbf w)$, if non-empty, has dimension $2\mathbf v^t\mathbf w - \mathbf v^t C\mathbf v$, where as usual $C$ denotes the Cartan matrix associated to the underlying graph of our quiver.

\subsection{Affine Quiver Varieties are Symplectic}

Let $(Q,I,H)$ be a quiver and $\mathbf v,\mathbf w$ be dimension vectors. Let $Y=Y(\mathbf v,\mathbf w)$, respectively $X= X(\mathbf v,\mathbf w)$, be the smooth symplectic quiver variety and affine quiver variety associated to the dimension vectors $\mathbf v$ and $\mathbf w$ and let $\pi\colon Y \to X$ denote the natural projective morphism.

We first investigate the affine symplectic variety $X$. Note that since the morphism $\pi\colon Y \to X$ is not in general birational, the following  is not obvious.

\begin{theorem}[Theorem~1.2 of \cite{BS}]
\label{prop:affine-is-symplectic}
An affine quiver variety $X$ is a symplectic variety.
\end{theorem}

It follows that every affine quiver variety is holonomic. In fact, one can give an explicit description of the symplectic leaves.  Recall from \cite{N98} that $X$ is stratified by stabilizer-type: a point $z\in X$ corresponds to a closed orbit in $\mu^{-1}(0)$. Let $(x,p,q) \in \mu^{-1}(0)$ be a point lying in this orbit and let $H$ denote its stabilizer in $G_\mathbf v$. Since the orbit of $(x,p,q)$ is closed the subgroup $H$ is a reductive subgroup of $G_\mathbf v$. If $(H)$ denotes the conjugacy class of the subgroup $H$ in $G_\mathbf v$ we say that $(H)$ is the stabilizer type of $z \in X$. Conversely, given a conjugacy class $(H)$ of reductive subgroups of $G_\mathbf v$ we let $X_{(H)}$ be the subset of all points of $X$ of stabilizer type $(H)$. This yields a finite\footnote{There are finitely many conjugacy classes of such reductive subgroups as can be seen by root combinatorics.} stratification $X = \bigsqcup_{(H)} X_{(H)}$. 

As mentioned in the introduction, \cite{BS} also give explicit criteria for when an affine quiver variety admits a symplectic resolution.

\begin{cor}
\label{Quiver holonomic}
The variety $X$ is a holonomic variety in the sense of Kaledin.  Its symplectic leaves are precisely the strata of the orbit stratification in the sense of \cite{N98}.
\end{cor}
\begin{proof}
Since $X$ is a quotient of the zero locus of the moment map, it has a natural Poisson structure. Now it is shown in \cite{Martino} that the strata of the orbit stratification of $X$ are precisely the symplectic leaves of $X$. Since an integral Poisson subvariety $Y$ of $X$ is the union of the symplectic leaves it contains, it follows that there is a leaf which is open and dense in $Y$, and so the Poisson bracket is generically nondegenerate on $Y$, that is, $X$ is holonomic as claimed.
\end{proof}

\subsection{Smooth Quiver Varieties are Symplectic Resolutions}
While it is not, in general, true that an affine quiver variety has a symplectic resolution, it is nonetheless the case that a smooth symplectic quiver variety $Y(\mathbf v,\mathbf w)$ {\em is} always a symplectic resolution---although not necessarily of the associated affine quiver variety $X(\mathbf v,\mathbf w)$.   Although this seems to be known to experts, it does not seem to be documented anywhere. 

To prove it, we use the slice theorem of Nakajima, which shows that if $x \in X$ then there is an open neighbourhood $U$ of $x$ and a slice $U_x$ so that $U = U_x\times T$ where $U_x$ is (analytically) isomorphic to an open neighbourhood of $0 \in X_x$, where $\pi_x \colon Y_x \to X_x$ is a smaller quiver variety associated to a possibly different quiver (which is explicitly described by Nakajima).  We obtain a diagram:

\xymatrix{
& & & & \pi^{-1}(U) \ar[r]\ar[d]_{\pi}  & T\times (\pi_x)^{-1}(U_x) \ar[d] \\
& & & & U \ar[r]^{=}&   T\times U_x.
}

In particular, the fiber $\pi^{-1}(x)$ is biholomorphic to $L_x$, the fibre over $0 \in X_x$. This is known, by \cite[Theorem 3.21]{N98}, to be homotopic to $Y_x$, the smooth quiver variety associated to the slice at $x$, and hence has the topology of a smooth connected affine variety of dimension $\dim_{\mathbb C}(X_x) = \dim(X)-\dim(X_{(H)})$ where $X_{(H)}$ is the stratum of $X$ containing $x$ (note that the connectedness of smooth symplectic quiver varieties is proved in general by Crawley-Boevey \cite{C01}).

\begin{prop}\label{smooth is resolution}
The smooth quiver variety $Y(\mathbf v,\mathbf w)$ is a symplectic resolution.
\end{prop}
As above, we will write $Y=Y(\mathbf v,\mathbf w)$ and $X = X(\mathbf v,\mathbf w)$.
In order to show that a smooth symplectic variety is a symplectic resolution, we must show that it is birational to its affinisation, and that the affinisation is a normal variety. Now while we have the natural projective morphism $\pi\colon Y \to X$, it may be that this morphism is \textit{not} the affinisation of $Y$ (see Example \ref{quiver example} below).
\begin{proof}[Proof of Proposition \ref{smooth is resolution}]
Since $Y$ is connected and $\pi$ is projective, its image $\pi(Y)$ is a closed irreducible subset of $X$. Moreover, since the morphism $Y\rightarrow X$ is projective and Poisson by construction, its image is a closed Poisson subscheme of $X$, in particular a union of orbit strata $X_{(H)}$ of the affine quotient.

By the proof in \cite[\S 3.2]{N00} of the slice theorem recalled above, the fibre $\pi^{-1}(x)$ over a point $x \in X_{(H)}$ is homotopic to an affine variety of (complex) dimension $\dim(Y)-\dim(X_{(H)})$ and hence since there are only finitely many strata, there must exist a stratum for which 
$\dim(Y) = \dim(X_{(H)})$ (see also \cite[Lemma 10.4]{N98}). Write $V:=X_{(H)}$ for this stratum. It follows $V$ is open dense in the image $\pi(Y)$, and $\pi$ has finite fibres over $V$. But then again by the slice theorem, the fibre over $x \in V$ is homotopic to $X_x$ which is connected by \cite{C01}, and hence $\pi$ is in fact bijective over $V$, and so $Y$ and $\pi(Y)$ are birational. As $Y$ is smooth, the morphism $\pi$ factors through the normalisation, $N$ say, of $\pi(Y)$, and hence the induced morphism $\tilde{\pi}\colon Y \to N$ is still birational (and again projective). Now the complement of $V$ in $\pi(Y)$ is of codimension $2$, since it is a union of the strata of $X$, thus the locus over which $\pi$ is not an isomorphism is of codimension $2$ in $\pi(X)$. Since $\pi(X)$ is thus nonsingular in codimension one, it follows that $\tilde{\pi}(X) \rightarrow \pi(X)$ is an isomorphism in codimension one, and thus so is $\tilde{pi}$. But then (identifying $V$ with its preimage in $N$) any function defined on $V$ extends uniquely to a function on all of $N$ by normality, and thus $\tilde{\pi}_*(\mathcal O_{Y}) = \mathcal O_{N}$, so that $N$ is the affinisation of $Y$. 

Finally to see that $Y$ is a symplectic resolution, we need to check that $N$ is a symplectic variety, that is, that $N$ carries a symplectic form on its smooth locus. However, the Poisson bivector field associated to the symplectic form on $V$ extends to all of $N$ by normality, since the complement of $V$ is of codimension $2$ in $N$.  But then its top exterior power is a section of the anticanonical bundle of the smooth locus $N^\text{reg}$ of $N$, and its vanishing locus is precisely the locus on which the Poisson structure is degenerate. The latter locus is either empty or a divisor in $N^\text{reg}$, and since we know the Poisson structure is nondegenerate on $V$, the locus is empty in $N^{\text{reg}}$. 
\end{proof}

\begin{remark}
Note that the preceeding proof shows slightly more than is asserted in the statement: the smooth symplectic variety $Y(\mathbf v,\mathbf w)$ is a symplectic resolution of the normalization of a Poisson integral subscheme of $X(\mathbf v,\mathbf w)$. 
\end{remark}

Finally, we note that a smooth quiver variety is always a {\em conical} symplectic resolution. Indeed, if we let $\mathbb G_m$ act on the vector space $M(\mathbf v,\mathbf w)$ with weight $1$, we obtain $\mathbb G_m$-actions on $X(\mathbf v,\mathbf w)$ and $Y(\mathbf v,\mathbf w)$ for which $\pi$ is equivariant.  The action has weight $2$ on the symplectic form on $Y(\mathbf v,\mathbf w)$. The following is then clear:

\begin{lemma}
The induced $\mathbb G_m$-action on $Y(\mathbf v,\mathbf w)$ gives $Y(\mathbf v,\mathbf w)$ the structure of a conical symplectic resolution.
\end{lemma}

\begin{remark}
Note that the $\mathbb G_m$-action need not be unique: in general, the automorphism group of a quiver variety may contain a higher rank torus, and many of its one-parameter subgroups may yield conical structures on the quiver variety.   Such group actions (or {\em flavor symmetries} of supersymmetric gauge theories) play a central role in the beautiful story of symplectic duality discovered by Braden-Licata-Proudfoot-Webster \cite{BLPW}.

It should also be noted that Nakajima in \cite{N98} uses a different $\mathbb G_m$-action which depends on a choice of orientation $\Omega$ of the quiver. For quivers with no oriented loops it is a conical action. See \cite[Proposition 3.22]{N98} for more details.
\end{remark}

\subsection{Universal Hamiltonian Reduction as the Versal Poisson Deformation}
The moment map for the action of $G_\mathbf v$ on $M(\mathbf v,\mathbf w)$ takes values in $\mathfrak{g}_{\mathbf{v}}^*$.  We have 
\[
\mathfrak z_\mathbf v = (\mathfrak g_\mathbf v/[\mathfrak g_\mathbf v,\mathfrak g_\mathbf v])^* \subseteq \mathfrak{g}_{\mathbf{v}}^*.
\]
We may thus consider the flat family of reductions $\mu^{-1}(\mathfrak z_\mathbf v)/\!\!/_\chi G_\mathbf v$ over $\mathfrak z_\mathbf v$, and it follows from the semiuniversal property that this deformation has an associated period map $\kappa \colon \mathfrak z_{\mathbf v} \to H^2(Y(\mathbf v,\mathbf w))$. 
It is known that the map $\kappa$ is actually given by the Chern character map: for details see \cite{L12}. 
\begin{theorem}[\cite{McGNKirwan}, Theorem 1.1] 
For any quiver variety the map 
\begin{displaymath}
H^*_{G_\mathbf v}(\star) \cong H^*_{G_\mathbf v}(E_V) \to H^*_{G_\mathbf v}(\mu^{-1}(0)^{\text{ss}}) \to H^*(Y(\mathbf v,\mathbf w))
\end{displaymath}
 is surjective.
 \end{theorem} 
\begin{cor}
The period map $\kappa$ is surjective. 
\end{cor}

Thus, up to identifying the kernel of $\kappa$, this yields a realization of the semiuniversal deformation.

\subsection{Quiver Varieties of Type ADE}
In this section, we provide refined information for type ADE quiver varieties.  
Thus, from now on, we will assume that the underlying graph of our quiver $Q$ is a simply-laced Dynkin diagram. As we noted before, the results in this subsection are presumably well known to experts. 

Let $P$ and $R$ be the weight and root\footnote{$Q$ would be more standard, but here we use $Q$ to denote a quiver.} lattices, respectively, of the corresponding semisimple Lie algebra, and let $\Phi$ denote the set of roots. As usual we let $\Phi^+$, $R^+$ and $P^+$ denote the positive elements in these sets, and we write $(\cdot,\cdot)$ for the pairing between $P$ and $R$.
It will be convenient to identify $\mathbf w$ with the element $\Lw = \sum_{i \in I} w_i\Lambda_i$ in $P^+$ the set of dominant weights, where $\{\Lambda_i\}_{i \in I}$ are the fundamental weights. Similarly we identify $\mathbf v$ with the element $\Qv = \sum_{i\in I} v_i\alpha_i$  of $R^+$, where $\{\alpha_i\}_{i \in I}$ are the simple roots. If $\nu =\sum_{i \in I} \nu_i\alpha_i \in R$ we write $\text{supp}_R(\nu) = \{i \in I: \nu_i \neq 0\}$, and similarly for $\mu = \sum_{i \in I} \mu_i \Lambda_i \in P$ we write $\text{supp}_P(\mu) = \{i \in I: \lambda_i \neq 0\}$.

Let $X(\mathbf v,\mathbf w)^{\text{reg}}$ be the locus of points in $X(\mathbf v,\mathbf w)$ corresponding to closed $G_\mathbf v$-orbits in $\mu^{-1}(0)$ with trivial stabilizer. (This is an open but possibly empty subset of $X(\mathbf v,\mathbf w)$.)
Note that, if we fix $\mathbf w$, then if $\mathbf v'$ has $v'_i \leq v_i$ we may choose $V'_i$ as a direct summand of $V_i$, and hence extending by zero yields a natural map $X(\mathbf v',\mathbf w) \to X(\mathbf v,\mathbf w)$ which is injective.

\begin{prop}
\label{Xstratification}
For an affine quiver variety $X$ attached to a simply-laced finite-type Dynkin diagram, the orbit strata (and hence the symplectic leaves) are given by
\[
X(\mathbf v,\mathbf w) = \bigsqcup_{\mathbf v'} X(\mathbf v',\mathbf w)^{\text{reg}},
\]
where the $\mathbf v'$ that occur as non-empty strata are those for which both $v'_i \leq v_i$ for all $i \in I$ and such that $\Lw - \alpha_{\mathbf v'} \in P^+$. 
\end{prop}
\begin{proof}
This is almost contained in \cite{N98}. Indeed it is shown in Remark $3.28$ of that paper that for quiver of type $ADE$, the symplectic leaves are all of the form $X(\mathbf v'\mathbf w)^{\text{reg}}$, and moreover Corollary $10.8$ therein shows that $X(\mathbf v,\mathbf w)^{\text{reg}}$ is nonempty if and only if $\Lw -\Qv \in P^+$, and $\Lw -\Qv$ is a weight of the irreducible representation $L(\Lw)$ of highest weight $\Lw$. But it is known that the weights $\Pi(\Lw)$ of the irreducible representation $L(\Lw)$ are, in the sense of \cite[\S 13.4]{H}, a saturated set with highest weight $\Lw$, and so, by Lemma B of \textit{op. cit.} a weight $\mu$ lies in $P^+\cap \Pi(\Lw)$ if and only if $\Lw - \mu \in Q^+$. Applying this to $\mu = \Lw - \Qv$ gives the result. 
\end{proof}

Finally we note that in the finite type case, $X(\mathbf v,\mathbf w)$ always has a symplectic resolution.

\begin{prop}\label{prop:res-exists}
Let $X(\mathbf v,\mathbf w)$ be a finite-type affine quiver variety. Then there is a $\mathbf v' \in \mathbb Z^I$ with $\alpha_{\mathbf v'} \leq \Qv$ such that $X(\mathbf v',\mathbf w)\cong X(\mathbf v,\mathbf w)$ and  $\pi_{\mathbf v'}\colon Y(\mathbf v',\mathbf w)\to X(\mathbf v',\mathbf w)$ is a symplectic resolution.
\end{prop}
\begin{proof}
By Proposition \ref{Xstratification} there is a $\mathbf v'$ with $X(\mathbf v',\mathbf w)^{\text{reg}}$ dense and open in $X(\mathbf v,\mathbf w)$. It follows that the natural injective morphism $i\colon X(\mathbf v',\mathbf w) \to X(\mathbf v,\mathbf w)$ is surjective, and hence bijective. Since by the main result of \cite{C03}, the variety $X(\mathbf v,\mathbf w)$ is normal, it follows that $i$ is an isomorphism. Corollary 10.11 of \cite{N98} shows that the projective morphism $\pi_\mathbf v$ is a semismall resolution of its image, and moreover all strata contained in the image are relevant. Since $\dim(X(\mathbf v',\mathbf w)) = \dim(Y(\mathbf v',\mathbf w))$ it follows that $\pi_{\mathbf v'}$ is dominant, and so the image is all of $X(\mathbf v,\mathbf w)$. Now the pullback of the symplectic form on $X(\mathbf v',\mathbf w)^\text{reg}$ is readily seen to coincide with the restriction of the symplectic form on $Y(\mathbf v',\mathbf w)$, so that it follows $\pi_\mathbf v$ is a symplectic resolution provided we know $X(\mathbf v',\mathbf w)^{\text{reg}}$ is the entire smooth locus.

To see this note that the Poisson bracket associated to the symplectic form extends to all of $X(\mathbf v',\mathbf w)$ by normality, and then by considering the top exterior power of the associated bivector field we see that its degeneracy locus within the smooth locus of $X(\mathbf v',\mathbf w)$ must either be a divisor or empty. Since it does not intersect $X(\mathbf v',\mathbf w)^{\text{reg}}$ it is empty.
\end{proof}

\begin{example}\label{quiver example}
The above allows us to see one way in which $\pi \colon Y(\mathbf v,\mathbf w)\to X(\mathbf v,\mathbf w)$ need not be a resolution: Indeed if we take $Q$ to be the quiver with one node, and $v=w=2$ then $Y(2,2)$ has dimension $0$ (and so is a point) while $X(2,2)$ is the rational singularity $\mathbb C^2/\{\pm 1\}$, and thus $\pi$ is not surjective. An explicit calculation shows that the zero fibre $\mu^{-1}(0)$ of the moment map has three irreducible components, two of which are picked out by stability conditions. These components each contain a free $G_\mathbf v$-orbit however, so that the invariant functions are constant on each, while on the third component there are no free orbits, and there are non-constant $G_\mathbf v$-invariant functions on it.
\end{example}

\section{Basics of Symplectic Springer Theory}
\label{Basics}
We lay out some general algebraic and topological features of symplectic Springer theory.  We let $k$ denote a coefficient field for constructible sheaves (and their cohomologies).
\subsection{Harish-Chandra and Springer Sheaves}
We maintain the notation of Diagram \eqref{main diagram}.  
\begin{definition}
The constructible complex 
\begin{displaymath}
\mathsf{HC} := \widetilde{\pi}_!k_{\mathcal{Y}}[\operatorname{dim}(\mathcal{Y})]
\end{displaymath}
 is the {\em symplectic Harish-Chandra sheaf} on $\mathcal{X}$.  The constructible complex 
 \begin{displaymath}
 \mathsf{Spr} := \pi_!k_{Y}[\operatorname{dim}(Y)]
 \end{displaymath}
  is the {\em symplectic Springer sheaf} on $X$.
\end{definition}
We collect some basic properties below. Let $\operatorname{Perv}_{\operatorname{sympl}}(X)$ denote the category of perverse sheaves on $X$ (with coefficients in $k$) which are smooth along the stratification by symplectic leaves.

\begin{theorem}
\label{Springer properties}
\mbox{} Let $W$ denote the symplectic Galois group.  
\begin{enumerate}
\item The complexes $\mathsf{HC}$ and $\mathsf{Spr}$ are perverse sheaves, which are semisimple when $\mathsf k$ has characteristic zero. Moreover $\mathsf{HC}$ is the intersection cohomology complex on $\mathcal X$ extending the local system on $\mathcal X^{\text{reg}}$ given by the regular representation of $W$ the symplectic Galois group.
\item We have $i_X^*\mathsf{HC} \cong \mathsf{Spr}$, where $i_X\colon X  \hookrightarrow \mathcal X$. 
\item We have a natural algebra isomorphism
\begin{displaymath}
k[W]\xrightarrow{a} \operatorname{End}(\mathsf{HC}).
\end{displaymath}
\item We get an adjoint pair of functors:
\begin{displaymath}
(-\otimes_{k[W]} \mathsf{Spr}): k[W]-\operatorname{mod} \leftrightarrows \operatorname{Perv}_{\operatorname{sympl}}(X): \operatorname{Hom}(\mathsf{Spr}, -).
\end{displaymath}

\end{enumerate}
\end{theorem}

\begin{proof}
Statement (1) is immediate from the semismallness of $\pi: Y\rightarrow X$ and the smallness of the map $\tilde{\pi}$ established in Proposition \ref{smallness}. Note that the semisimplicity statement about the local system will follow from the proof of $(3)$ below.  Statement (2) follows from proper base change.  

For statement (3), let $B_X^{\operatorname{reg}}$ denote the open subset of the base $B_X$ of the versal deformation over which $\mathcal{X}$ has smooth fibers.  Write $\mathcal{X}^{\operatorname{reg}}$ for $\mathcal{X}\times_{B_X} B_X^{\operatorname{reg}}$ and let 
$j: \mathcal{X}^{\operatorname{reg}}\hookrightarrow \mathcal{X}$ denote the inclusion.  Note that $\widetilde{\pi}^{-1}(\mathcal{X}^{\operatorname{reg}})
\rightarrow \mathcal{X}^{\operatorname{reg}}$ is an unramified Galois $W$-cover.  Now, 
observe that smallness of $\widetilde{\pi}$ implies that $\mathsf{HC} = j_{!\ast}\widetilde{\pi}_*k_{\widetilde{\pi}^{-1}(\mathcal{X}^{reg})}[-\operatorname{dim}(\mathcal{X})]$.  We use the following standard lemma. 
\begin{lemma}
For a dense open subset $U\xrightarrow{\iota} X$ and a local system $\mathcal{L}$ on $U$, 
\begin{displaymath}
\operatorname{End}_X(\iota_{!\ast}\mathcal{L}) = \operatorname{End}_U(\mathcal{L}).
\end{displaymath}
\end{lemma}
\begin{proof}[Proof of Lemma]
We have homomorphisms
\begin{displaymath}
 \operatorname{End}_U(\mathcal{L}) \longrightarrow \operatorname{End}_X(\iota_{!\ast}\mathcal{L}) \longrightarrow \operatorname{End}_U(\mathcal{L}),
 \end{displaymath}
 the first given by $(-)_{!\ast}$ and the second by restriction to $U$.  The composite is the identity; hence it suffices to show that the latter map is injective.  
If $\phi\in \operatorname{End}_X(\iota_{!\ast}\mathcal{L})$ and $\phi|_U = 0$, it follows that $\phi(\iota_{!\ast}\mathcal{L})$ is supported on $X\smallsetminus U$.  But
the intermediate extension $\iota_{!\ast}\mathcal{L}$ is the unique extension that has no sub- or quotient object supported on $X\smallsetminus U$, so $\operatorname{im}(\phi)=0$, as required.
\end{proof}
It follows that
\begin{displaymath}
\operatorname{End}(\mathsf{HC}) \cong \operatorname{End}\big(\widetilde{\pi}_*k_{\widetilde{\pi}^{-1}(\mathcal{X}^{reg})}\big)\cong k[W].
\end{displaymath}
This defines the isomorphism $a$.

Finally note that the existence of the adjunction claimed in (4) is formal, provided we prove that $\mathsf{Spr}$ is smooth along the stratification of $X$ by symplectic leaves. This is a consequence of the nearby cycles description of the Springer sheaf which we will give in Section \ref{nearby cycles}, since the nearby cycles sheaf is constructible with respect to a Whitney stratification of the exceptional fibre that satisfies the $a_f$ condition. 

\end{proof}

\begin{remark}
It is perhaps natural (or possibly abusive) to call the kernel of the functor $\operatorname{Hom}(\mathsf{Spr}, -)$ the subcategory of {\em cuspidal symplectic perverse sheaves} on $X$.  It seems likely to be hard to characterise this category in general (or even tell when it is nonzero).
\end{remark}
\begin{remark}
As we will see in examples later in this section, the stronger version of assertion (3), which readers familiar with the classical theory might have expected to see, namely that the composition of the map $a$ with the natural map $\operatorname{End}(\mathsf{HC})\rightarrow \operatorname{End}(\mathsf{Spr})$ is an isomorphism, is not true in general.  Neither, then, is it true that the induction functor from $k[W]$-modules to $\operatorname{Perv}_{\operatorname{sympl}}(X)$ is faithful. Indeed we will see later in this section examples which show that the map from $k[W]$ to $\text{End}(\mathsf{Spr})$ need neither be injective nor surjective. The following Lemmas will however imply a weaker statement, namely that $W$ embeds into the automorphisms of $\mathsf{Spr}$.
\end{remark}

Next we note that the deformation $\tilde{\pi}\colon \mathcal Y \to \mathcal X$ gives us a more elementary way of constructing an action of $W$ on the cohomology of $Y$. 
We first remark that the endomorphism algebra of the sheaf $\mathsf{Spr}$ has a natural geometric interpretation.
Let $Z = X\times_Y X$ be the fibre product of $X$ with itself over $Y$ (the ``symplectic Steinberg variety''). Its Borel-Moore homology has a natural convolution structure preserving the middle-dimensional piece $H^{\text{mid}}_{BM}(Z)$; the latter has a natural basis given by the irreducible components of $Z$. 
\begin{lemma}\label{Springer-via-BM}
We have 
\[
\text{End}(\mathsf{Spr})\cong H^{\operatorname{mid}}_{BM}(Z).
\]
\end{lemma}
\begin{proof}
The argument is standard: see, for example, Section 3.4 of \cite{CG}.
\end{proof}

In terms of the identification $\text{End}(\mathsf{Spr}) \cong H^{\text{mid}}_{BM}(Z)$, the map $W$ to  $H^{\text{mid}}_{BM}(Z)$ can be seen as a specialization-of-cycles map, similar to the construction of Springer representations in 
\cite{CG}.

Since we are assuming that our symplectic resolution $\pi\colon Y\to X$ is conic, the fixed point locus $Z = Y^{\mathbb G_m}$ is a smooth projective variety contained in $L=\pi^{-1}(0)$ (where as usual we write $0$ for the cone point of $X$). Moreover, since this $\mathbb G_m$ action contracts $Y$ to $Z$, it follows we have isomorphisms in cohomology $H^*(Y) \cong H^*(L) \cong H^*(Z)$. (Note however that in general $L$ is not smooth and $Z \subsetneq L$.)

\begin{definition}
At a regular point $b \in B_Y$, that is, at a point outside the discriminant locus, the map $\tilde{\pi}$ is $W$-sheeted local isomorphism and we get an isomorphism between each fibre of $\psi_Y$ over the orbit $W\cdot b \subset B_Y$ and $\psi_X^{-1}(\pi(b))$. This induces an action of $W$ on the cohomology of the fibres (for example induced by the action of $W$ on the complex $(\psi_Y)_!(k)$). But the deformation $\mathcal Y$ is topologically trivial, so the cohomologies of the fibres of $\psi_Y$ are canonically identified, and hence we get an action on the cohomology of $Y$. We will denote this action by $\rho\colon W\to GL(H^*(Y))$. Note that the action on the fibres at regular points is precisely the one inducing the morphism $a$ as in Theorem \ref{Springer properties}.
\end{definition}

\begin{lemma}
The restriction of $\rho$ to $H^2(Y)$ agrees with the action of $W$ on $B_Y = H^2(Y)$.
\end{lemma}
\begin{proof}
This follows from the construction of the period map, which is given by taking the homology class of the symplectic form of the fibres of $\psi_Y$. See for example Step 2 in the proof of Theorem 1.1 in \cite{Na2} for details.
\end{proof}

\begin{lemma}
\label{naive identification}
The action of $W$ on $H^*(Y)$ via $\rho$ coincides with the action of $W$ on $\pi_!(k_Y)_{0}= H^*(L)$ via the isomorphism $H^*(L)\cong H^*(Y)$ given by the $\mathbb G_m$-contraction.
\end{lemma}
\begin{proof}
We have the following diagram:

\xymatrix{
& & & & \mathcal Y \ar[r]^{\psi_Y} \ar[d]_{\tilde{\pi}} & B_Y \ar[d]^{\pi} \\
& & & & \mathcal X \ar[r]_{\psi_X} & B_X \ar[r]^{=} & B_Y/W
}
As already noted, the $\mathbb G_m$-action on $X$ extends to the deformations $\mathcal X$ and $\mathcal Y$, and hence the sheaf $\tilde{\pi}_!(k)$ is equivariant with respect to the $\mathbb G_m$-action (since the constant sheaf $k_Y$ clearly is, and the morphism $\pi$ is equivariant). It follows (see for example \cite[Lemma 5.6]{Ka}) that the natural morphism $H^*(\mathcal X,\tilde{\pi}_!(k))\to \tilde{\pi}_!(k)_{0}$ is an isomorphism. Now
$
H^*(\mathcal Y,k) = H^*(\mathcal X,\pi_*(k)) = H^*(\mathcal X,\tilde{\pi}_!(k))).
$
But the map $\psi_Y$ is a topologically trivial fibration, so the restriction-to-fibres maps $H^*(\mathcal Y,k) \to H^*(\mathcal Y_t,k)$ are all isomorphisms. It follows readily that the composition $H^*(Y, k) \cong H^*(\mathcal Y, k) \to \tilde{\pi}_!(k)_{0} \cong \pi_!(k)_0$ is $W$-equivariant as required.
\end{proof}

\begin{cor}\label{cor:action}
The symplectic Galois group $W$ acts faithfully on the sheaf $\mathsf{Spr}$. Moreover this action yields, for any $x \in X$, a natural action of $W$ on the cohomology $H^*(\pi^{-1}(x))$.
\end{cor}
\begin{proof}
The second assertion follows immediately from the action of $W$ on $\mathsf{Spr}$ induced by the restriction of the action on $\mathsf{HC}$, since the cohomology in question is just the stalk of $\mathsf{Spr}$ at $x\in X$. For the first part, note that it is enough to check that the induced action on some stalk cohomologies of $\mathsf{Spr}$ is faithful; but this follows for the stalk at $0 \in X$ by the preceding Lemmas. Indeed, these show that action of $W$ on $H^2(\pi^{-1}(0))$ is just that of $W$ on a generic fiber of $B_Y\rightarrow B_X$, which is faithful by construction.
\end{proof}

\subsection{Stratifications and Thom's $a_f$ Condition for Poisson Deformations}
\label{nearby cycles}
Let $X\subset U \subset \mathbb C^n$ be an analytic subset of $U$ an open set in $\mathbb C^n$, and suppose $f\colon X \to \mathbb C$ is an analytic map, and $x \in f^{-1}(0)$. For $\delta \ll \varepsilon \ll 1$ the restriction of $f$ to $f^{-1}(\dot{D}_\delta\backslash\{0\})\cap B_\varepsilon(x)$ is a topological fibration. Here $\dot{D}_\delta$ is the open disc of radius $\delta$ centered at $0$, and $B_\varepsilon (x)$ denotes the closed ball of radius $\varepsilon$ around $x$ (with respect to some Hermitian metric say). 

This was first proved by Le \cite{Le}, generalizing results of Milnor and others. If we take $f\colon X \to Y$ where $\dim(Y)$ is greater than $1$, the corresponding statement is false.

\begin{example}
Let $f\colon \mathbb C^3 \to \mathbb C^2$ be given by $f(x,y,z) = (x^2-y^2z, y)$. It can be checked that $f$ is flat, however it is not a topological fibration in any ball around $0 \in \mathbb C^3$: indeed it is easy to see that some fibres are connected and some are not. (See the introduction to \cite{Sa} for a detailed discussion of this example.)
\end{example}

However, one can impose a condition, going back to R. Thom, on the map $f$ which ensures that an analogue of Le's result holds. In order to describe this we will need some basic material from stratification theory, as dicussed for example in \cite{Mather}. Recall that a holonomic Poisson variety in the sense of \cite{Ka} is an integral Poisson variety $X$ for which every integral Poisson subscheme $Z$ is generically nondegenerate, that is, generically on $Z$ the Poisson bivector is nondegenerate when restricted to $Z$. It is shown in \cite{Ka} that any symplectic variety is holonomic.

\begin{lemma}
\label{Whitneystrata}
Let $X$ be a symplectic variety. Then the natural stratification of $X$ given by Theorem \ref{Kaledin strata} is a Whitney stratification.
\end{lemma}
\begin{proof}
To see this we use the basic fact which establishes the existence of Whitney stratifications: given any analytic variety $X$ with a partition $X = \bigsqcup_{i \in I} X_i^0$, and strata $X^0_i,X^0_j$ satisfying $X^0_j \subset X_i = \overline{X^0_i}$, the set of points $x \in X^0_j$ which do not satisfy the Whitney $a$ and $b$ condition is an analytic set of dimension strictly smaller than $X^0_j$. We call such points irregular. 

In the case where $X$ is a symplectic variety, the strata $\{X_i^0: i \in I\}$ of Theorem \ref{Kaledin strata} are precisely the symplectic leaves of $X$, and can be described explicitly as follows: the open leaf $X_0$ gives the regular locus, then if $\Sigma$ is the singular locus we take the connected components of the smooth locus of $\Sigma$, and so on. This follows from the fact that a smooth holonomic variety is in fact symplectic (see Lemma $1.4$ in \cite{Ka}). To verify the Whitney conditions, note that if $x \in X_j^0\subset X_i$ $x$ is a point of irregularity for a pair of strata $X_i^0,X_j^0$, then its orbit under any Hamiltonian flow (which preserves both $X_i^0$ and $X_j^0$) will consist entirely of points of irregularity. It follows that the set of points of irregularity is a union of symplectic leaves in $X$. But then if the irregular locus is non-empty it must be all of $X_j^0$, contradicting the fact that it has smaller dimension than $\dim(X_j^0)$. 
\end{proof}

\begin{remark}
The iterative procedure yielding the symplectic leaves in the case of a holonomic variety always produces a partition of a variety into smooth locally closed subvarieties, but it does not normally yield a Whitney stratification.
\end{remark}

\begin{definition}
Let $M$ be a connected smooth complex analytic variety of dimension $n$ and suppose that $f\colon M \to U$ is an analytic map, where $U \subset \mathbb C^r$ is a neighborhood of $0 \in \mathbb C^r$. Let $U^{\text{reg}}$ be the set of regular values of $f$ and $U^{\text{sing}}$ the set of non-regular values. We assume that $U^{\text{sing}}$ is a proper analytic subvariety of $U$. Let $M^{0}= f^{-1}(U^{\text{reg}})$, a smooth submanifold of $M$. Suppose that $E =f^{-1}(0)$ is a equipped with a Whitney stratification $\mathcal S$. We will say that $(f,M,\mathcal S)$ satisfies Thom's $A_f$ condition if for each $S\in \mathcal S$, whenever a sequence of points $x_i \in M^0$ converging to $y \in S$, is such that the limit 
\[
T = \lim_{i \to \infty} T_{x_i}(f^{-1}(f(x_i)) \subset T_yM,
\]
exists, we have $T_yS \subseteq T$. 
\end{definition}

When $(f,M,S)$ satisfies the $A_f$ condition, then it has been shown by Grinberg that it makes sense to talk about the nearby cycles sheaf of $f$. Indeed we have the following proposition.

\begin{prop}
(\cite[Proposition 2.4]{Gr})
If $(f,M,S)$ satisfies the $A_f$ condition there is a well-defined sheaf $P_f$ of nearby cycles of $f$ supported on $E$ which is constructible with respect to the stratification $\mathcal S$. If $B_\epsilon$ is a small ball around $0 \in U$, then $\pi_1(U^{\text{reg}})\cap B_\epsilon$ acts on $P_f$ by monodromy. 
\end{prop}

The sheaf $P_f$ can be computed by pulling back along an analytic arc in $U$ and forming the usual one-dimensional nearby cycles sheaf. 
\begin{example}
In the case of $T^*\mathcal B$, the resolution of the nilpotent cone, the symplectic deformation space $H^2(T^*\mathcal B, \mathbb C)$ can be identified with the universal Cartan subalgebra $\mathfrak h$, and $W$ is the (usual) Weyl group of the algebraic group $G$. The stratification is then by the root hyperplanes, and the stratification of $\mathfrak h/W$ given by the corresponding orbits.
\end{example}

We are now ready to give our second construction of the Springer sheaf along with the action of the symplectic Galois group $W$. We keep the notation of Diagram \eqref{main diagram} and of Theorem \ref{Namikawa discriminant}. Set $U = (H^2(Y,\mathbb C) \backslash \mathcal D)/W$.

\begin{lemma}
The triple $(\theta_X,\mathcal X,\mathcal S)$ satisfies the $A_f$-condition for the morphism $\theta_X\colon \mathcal X' \to H^2(Y,\mathbb C)/W$.
\end{lemma}
\begin{proof}
The claim about $\mathcal U$ follows from \cite{N3}, see Theorem \ref{Namikawa discriminant}. If $x \in S$ is a point in a symplectic leaf, and $\eta \in T_xS$, then since $S$ is an integral Poisson subscheme, we may write the derivation given by $\eta$ via the Poisson structure as $g \mapsto \{f_\eta,g\}(x)$ for some function $f_\eta$. Now picking an extension of this function to a function $\tilde{f}_\eta$ on the variety $\mathcal Y$ we see that if $(y_n)$ is a sequence of points in $\mathcal U$ tending to $x$, since the fibre of $\tilde{\pi}$ at each $y_n$ is smooth symplectic, the derivation $\eta_n = \{\tilde{f}_\eta,\cdot\}(y_n)$ is tangent to $\tilde{\pi}^{-1}(y_n)$ and clearly 
\[
\lim_{n \to \infty} \eta_n = \eta,
\]
so that the tangent space $T_x S$ lies in the limit of the tangent spaces of the fibres of $\tilde{\pi}$ as required.
\end{proof}

The previous Lemma shows that $\theta_X$ is a morphism without blow-up, so we may define a nearby cycles sheaf $\mathcal P = \psi_{\theta_X}(\mathbb C_{\mathcal X'})$ carrying an action of the fundamental group of $U= (H^2(Y,\mathbb C)\backslash \mathcal D)/W$ where $\mathcal D$ denotes the discriminant of the map $\theta_Y$.

\begin{theorem}\label{nearby cycles2}
 We have $\mathcal P \cong \pi_!(\mathbb C_{Y})$ and the action of the braid group $\pi_1(U)$ factors through the symplectic Galois group $W$. 
\end{theorem}
\begin{proof}
Let us first see that $\mathcal  P$ is isomorphic to $\pi_!(\mathbb C_Y)$: pick a line $L\subseteq H^2(Y,\mathbb C)$ which intersects $\mathcal D$ only at $0 \in H^2(Y,\mathbb C)$ (this is possible since $\mathcal D$ is a finite union of hyperplanes) and let $C = q(L)$ where $q\colon H^2(Y,\mathbb C)\to H^2(Y,\mathbb C)/W$ is the quotient map.  We may construct $\mathcal P$ by forming the (one-dimensional) nearby cycles for the restriction $\theta_C
\colon \mathcal X_C \to C$ of the map $\theta_X$ to $\theta_X^{-1}(C)$. Let $\theta_L\colon \mathcal Y_L \to L$ denote the restriction of $\theta_Y$ to $L$, so that we have a commutative diagram

\xymatrix{
& & & & & \mathcal Y_L \ar[d]_{\theta_L} \ar[r]^{\tilde{\pi}} & \mathcal X_C \ar[d]^{\theta_C} \\
& & & & & L \ar[r]^{q_{|L}} & C.
}

Since the map $\tilde{\pi}$ is proper, the nearby cycles construction commutes with push-forward by $\tilde{\pi}$, that is,
\[
\psi_{\theta_C}(\tilde{\pi}_!(\mathbb C_{\mathcal Y})_{|\theta_Y^{-1}(L)}) \cong \pi_!(\psi_{L}(\mathbb C_{\theta_Y^{-1}(L)})). 
\] 
Now it is known that the map $\mathcal Y \to H^2(Y,\mathbb C)$ is a topologically trivial fibre bundle, hence $\psi_g(\mathbb C_{f^{-1}(L)}) \cong \mathbb C_Y$. On the other hand, since $\tilde{\pi}$ is an isomorphism over $\mathcal U = \mathcal X'\backslash f^{-1}(\mathcal D)$ it follows $\tilde{\pi}_!(\mathbb C_{\mathcal Y})_{|f^{-1}(L\backslash\{0\}} \cong \mathbb C_{\theta_X^{-1}(C\backslash\{0\})}$ and so $\psi_{\theta_C}(\tilde{\pi}_!(\mathbb C_{\mathcal Y})) \cong \mathcal P$. It follows that $\mathcal P = \pi_!(\mathbb C_{Y})$ as claimed. Moreover the topological triviality of the fibre bundle $\theta_Y$ shows that the action of the fundamental group on $\mathcal P$ must factor through $W$ as required.
\end{proof}

\begin{cor}
The Springer sheaf $\mathsf{Spr} = \pi_!(\mathbb C)$ is constructible with respect to the strafication of $X$ by its symplectic leaves.
\end{cor}
\begin{proof}
This follows immediately from the isomorphism $\mathsf{Spr}\cong \mathcal P$ and the general fact that the nearby cycles sheaf is constructible with respect to a Whitney stratification satisfying the $A_f$-condition.
\end{proof}

\begin{remark}
The previous Corollary shows that the cohomologies of the fibres of the symplectic resolution are locally constant on the symplectic leaves.
\end{remark}

Finally, we show that the monodromy action of the symplectic Galois group given by the nearby cycles construction is the same as that given by taking the intermediate extension.
\begin{prop}\label{monodromy}
The monodromy action of $W$ on $\pi_!(\mathbb{C}_Y)$ induced by the isomorphism $\mathcal P \cong \pi_!(\mathbb C)$ coincides with the action of $W$ induced by the intermediate extension.
\end{prop}
\begin{proof}
First note the action of $W$ on the stalk cohomology of $\pi_!(\mathbb C_Y)$ at $0 \in X$ is faithful, and moreover, the stalks of the perverse sheaves occuring in the Springer sheaf are all non-vanishing at the cone point. It is therefore enough to check the actions coincide on this stalk cohomology. But it is easy to see that the naive action of $W$ on $H^*(Y) = H^*(L)$ (where $L=\pi^{-}(0)$ as above) defined in Section \ref{Basics} coincides with the monodromy action on $\mathcal P$.
\end{proof}

\subsection{Examples and a Question}\label{sec:ex-conj}
We continue with a discussion of some examples.  For simplicity, we take $k=\mathbb C$.

\begin{example}
The simplest cases of symplectic singularities are the Kleinian singularities. Let $\Gamma$ be a finite subgroup of $\text{SL}_2(\mathbb C)$, and let $S = \mathbb C^2/\Gamma$ be the corresponding quotient singularity, and let $\pi \colon Y\to S$ be the minimal resolution.
Now $\mathsf{Spr} = \pi_!(\mathbb C_Y[2])$ is a semisimple perverse sheaf. The only sheaves which can occur are the constant sheaf (which is perverse in this case because $S$ is rationally smooth) and the skyscraper sheaf supported at the cone point. Using proper base change to compute the cohomology of the fibres, this shows that $\mathsf{Spr} \cong \mathbb C_{S}[2]\oplus \mathbb C_{0}^{\oplus r}$ where $r$ is the number of components in the exceptional fibre. The action of the Weyl group is then via the trivial representation on $\mathbb C_S[2]$ and via the reflection representation on the summand supported on singular point.  

Note this example shows that in general one cannot expect that the map from $\mathbb C[W]$ to $\text{End}(\mathsf{Spr})$ is injective (even thought the map $W\to \text{Aut}(\mathsf{Spr})$ is injective).
\end{example}

\begin{example}\label{ex:Hilb}
Let $S$ be a smooth affine complex algebraic surface with trivial canonical bundle, and let $X= \text{Sym}^n(S)$ be its $n$-th symmetric product. It is well known that if $Y = \text{Hilb}^n(S)$ then the Hilbert-Chow morphism from $f\colon Y\to X$ is a resolution of singularities of $X$, making $Y$ a symplectic resolution, where the symplectic structures are induced by that on $S$. The symmetric product $X$ is naturally stratified into locally closed pieces $X_\nu$ given by partitions $\nu$ of $n$ (the partition measures the multiplicities of points) and these are evidently integral Poisson subvarieties. Moreover, as in \cite{GS} for example, we have $\dim(X_\nu) = 2n-2\ell(\nu)$ where $\ell(\nu)$ is the length of the partition $\nu$ (\textit{i.e.} the number of parts in the partition). It follows that the only codimension $2$ stratum is that labelled by $(2,1^{n-2})$. In order to calculate the symplectic Galois group it is thus enough to calculate a slice through this stratum (and check monodromies). However, in our case since we know it must be a Kleinian singularity, and moreover the description of the fibres of the Hilbert-Chow morphism show that the resolution of the slice has exceptional fibre consisting of a single $\mathbb P^1$ (the punctual Hilbert scheme of two points) the slice must be a type $A_1$ singularity and hence $W=S_2$, the symmetric group on two letters.

Now if $\{\beta_k\}_{k=0}^{4}$ denote the Betti numbers of the surface $S$, it follows easily from the results of \cite{GS} that 
\[
\dim\big(H^2(\text{Hilb}^n(S),\mathbb Q)\big) = \beta_2 + 1 + {\beta_1+1 \choose 2}.
\]
The nontrivial element of the symplectic Galois group $W$ acts as a reflection: indeed the proof of the main theorem in \cite{Na2} shows that as $W$-spaces we have $H^2(X)\cong H^2(Y)\times_{\mathfrak h/W} \mathfrak h$ where $\mathfrak h$ is the deformation base of the $A_1$-minimal resolution, and $H^2(X)$ maps to $\mathfrak h$ by the pull-back from the inclusion of the preimage in $\text{Hilb}^n(S)$ of $X_{(1^n)}\sqcup X_{(2,1^n)}$.

The Springer sheaf in this situation has been studied by Goettsche and Soergel \cite{GS}. 
They show that 
\[
\mathsf{Spr} \cong \bigoplus_{\nu \in \mathcal P_n} IC(\bar{X}_\nu),
\]
where $\mathcal P_n$ denotes the set of partitions of $n$, and $X_\nu$ is the symplectic leaf corresponding to the partition $\nu$. It follows that $\text{End}(\mathsf{Spr}) \cong \mathbb C^{p(n)}$, where $p(n) = |\mathcal P_n|$ is the number of partitions of $n$. Note that this example (for $n>2$) shows that the map $\mathbb C[W] \to \text{End}(\mathsf{Spr})$ need not be surjective. 
\end{example}

In general, the Springer sheaf always decomposes as a sum 
\[
\mathsf{Spr} \cong \bigoplus_{i \in I} IC(S_i,\mathcal{L}_i)\otimes V_i,
\]
where $i\in I$ runs over a finite set of local systems supported on symplectic leaves of $X$, and each $V_i$ is a finite-dimensional vector space. The $V_i$ thus become representations of $W$. 
The following question was suggested to us by Gwyn Bellamy.
\begin{question}
When are the representations $V_i$ irreducible?
\end{question}

Note that, as in Example \ref{ex:Hilb}, it is not true that the representations which occur need all be distinct (as is the case in the classical Springer theory). By studying resolutions of slices in the affine Grassmannian, we will later show that the representations $V_i$ are not always irreducible.

\section{Weyl groups and Symplectic Galois Groups of Quiver Varieties}\label{sec:quiver-Weyl}
We now compute the symplectic Galois groups for type $ADE$ quiver varieties. Nakajima \cite{N98} has shown how to construct highest weight representations in the cohomology of quiver varieties, where each variety $Y(\mathbf v,\mathbf w)$ corresponds to the weight space $\Lw - \Qv$. It turns out that if $W$ is the Weyl group attached to the $ADE$ quiver $Q$, then the symplectic Galois group is simply the parabolic subgroup of the Weyl group associated to this weight.  The corresponding can fail for $Q$ of affine type: see Example \ref{affine example}.

It follows that our Galois group action produces actions of the Weyl group on the cohomology of fibres of $\pi$ corresponding to zero weight spaces of irreducible representations of the associated Lie algebra. Such actions have been constructed by Lusztig, Nakajima and Maffei. In these constructions, the actions of simple reflections are defined, and it requires some effort to establish the braid relations. Our construction has the advantage of yielding the group action directly from topology, avoiding elaborate calculations. Of course, their work yields more explicit information, thus our approach should be seen as complementing theirs.

Let us introduce some notation: If $\lambda \in P$ then we write 
\[
\Phi_\lambda = \{\alpha \in \Phi: (\lambda,\alpha)=0\}
\]
for the parabolic subroot system associated to $\lambda$, and $\Phi^+_\lambda$ for the positive roots in $\Phi_\lambda$. Let us denote by $\Phi_{\lambda}^{\text{max}}$ the roots in $\Phi_\lambda$ which are maximal in $\Phi_\lambda$ with respect to the partial order on $R^+$.

\begin{lemma}
\label{codim 2 strata}
Let $X(\mathbf v,\mathbf w)$ be a quiver variety such that $X(\mathbf v,\mathbf w)^{\text{reg}} \neq \emptyset$. Then if $\mu = \Lw - \Qv$, the codimension two strata of $X(\mathbf v,\mathbf w)$ are indexed by the roots in $\Phi_{\mu}^{\text{max}}$.  
\end{lemma}
\begin{proof}
Recall that $X(\mathbf v',\mathbf w)^{\text{reg}} \neq \emptyset$ if and only if $\Lambda_\mathbf w - \alpha_{\mathbf v'}\in P^+$, and this is a stratum of $X(\mathbf v,\mathbf w)$ if $\alpha_{\mathbf v'}\leq \alpha_{\mathbf v}$. Let $\alpha_{\mathbf v} = \alpha_{\mathbf v'} + \beta$, and let $\mu = \Lambda_\mathbf w - \alpha_{\mathbf v}$. Then if $\Lambda_{\mathbf w}- \alpha_{\mathbf v'} \in P^+$, the dimension formula for quiver varieties shows that $X(\mathbf v',\mathbf w)^{\text{reg}}$ is a codimension two stratum precisely when
\[
2(\mu,\beta)+(\beta,\beta) = 2.
\]
Since $\mu \in P^+$ and $\beta \in R^+$, we must have $(\mu,\beta)=0$ and $(\beta,\beta)=2$, that is, $\beta \in \Phi_{\mu}$, the parabolic subsystem of $\Phi$ corresponding to $\mu$. Thus the codimension two strata are labelled by roots $\beta\in \Phi^+_{\mu}$ for which $\mu + \beta \in P^+$. 

We now claim that this set is exactly $\Phi^{\text{max}}_{\mu}$. Indeed if $\beta$ is such that $\mu+\beta \in P^+$, then for any $\alpha \in \Phi_{\mu}^+$ we have $0 \leq (\mu+\beta, \alpha) = (\beta,\alpha)$, and hence (as our root system is simply-laced, so root strings have length at most $2$) it follows that $\beta \in \Phi_{\mu}^{\text{max}}$. 

Conversely, if $\beta \in \Phi_{\mu}^{\text{max}}$, and $\alpha_i \in \Phi$ is a simple root, consider $(\mu+\beta,\alpha_i)$. If $\alpha_i \in \Phi_{\mu}$ then $(\beta,\alpha_i)\geq 0$ since $\beta$ is maximal (otherwise $s_{\alpha_i}(\beta) = \beta+ \alpha_i \in \Phi_{\mu}$) and so $(\mu+\beta,\alpha_i)\geq 0$.  If $\alpha_i \notin \Phi_{\mu}$, then since $\mu \in P^+$ we must have $(\mu,\alpha_i)\geq 1$, and since our root system is simply-laced, $(\alpha_i,\beta)\geq -1$, and hence $(\mu+\beta,\alpha_i) \geq 0$ in this case also. 
\end{proof}

\begin{lemma}
\label{codim 2 Weyl groups}
Suppose that $X(\mathbf v,\mathbf w)$ is as above, and $\beta \in \Phi_{\mu}^{\text{max}}$. Then the slice through the stratum associated to $\beta$ by Lemma \ref{codim 2 strata} is isomorphic to the Kleinian singularity indexed by the subroot system $\Phi(\beta) = \{\alpha \in \Phi: \alpha \leq \beta\}$, an indecomposable subsystem of $\Phi$.
\end{lemma}
\begin{proof}

By \cite{N00}, Nakajima constructs an (analytic) slice through a stratum $X(\mathbf v',\mathbf w)^{\text{reg}}$ which is isomorphic to $X(\mathbf v-\mathbf v',\mathbf w - C\mathbf v')$. In our case it follows that the slices are isomorphic to $X(\mathbf v_0,\mathbf w_0)$ where $\alpha_{\mathbf v_0} = \beta$ and $\Lambda_{w_0} = \mu+\beta)$.  It is known that the support $\text{supp}_R(\beta)$ of a root $\beta$ is always a connected subdiagram of the Dynkin diagram. The condition that $(\beta,\mu)=0$ shows that $X(\mathbf v_0,\mathbf w_0)$ is isomorphic to the quiver variety $X(\mathbf v_0',C'(\mathbf v_0'))$ for the quiver associated to the subdiagram $Q'$ given by $\text{supp}_R(\beta)$. (Here $C'$ is the Cartan matrix attached to $Q'$ and $\mathbf v_0'$ is just $\mathbf v_0$ viewed as a dimension vector for $Q'$.) It is well known that this quiver variety is the Kleinian singularity associated to the Weyl group of type $Q'$, or equivalently, the root system $\Phi(\beta)$. 
\end{proof}

\begin{remark}
The arguments above can also be used to show that the two-dimensional quiver varieties associated to an $ADE$ quiver $Q$ are precisely the Kleinian singularities attached to Dynkin diagrams which occur as subdiagrams of $Q$.
\end{remark}

\begin{theorem}
\label{QuiverWeylgroup}
Suppose $\mathbf v,\mathbf w$ are as above. Let $W_\mu$ be the Weyl group of the subroot system $\Phi_{\mu}$. The symplectic Galois group associated to the quiver variety $X(\mathbf v,\mathbf w)$ is $W_\mu$. 
\end{theorem}
\begin{proof}

We have seen that the codimension two strata are labelled by the set $\Phi_{\mu}^{\text{max}}$. Each $\beta \in \Phi_{\mu}^{\text{max}}$ is associated to an indecomposable subroot system $\Phi(\beta)$ of $\Phi_{\mu}$. Lemma \ref{codim 2 Weyl groups} shows that the slice through the associated stratum is the Kleinian singularity associated to the subroot system, which has symplectic Galois group $W(\Phi(\beta))$, the Weyl group of the root system $\Phi(\beta)$. Now Namikawa's description of the symplectic Galois groups shows that the group attached to $X(\mathbf v,\mathbf w)$ is 

\[
W = \prod_{\beta \in \Phi^{\text{max}}_{\mu}} W(\Phi(\beta))^{\sigma_\beta}
\]
where $\sigma_\beta$ is a diagram automorphism of the Dynkin diagram induced by monodromy coming from the stratum. Thus it remains to show that this monodromy must be trivial. We show this in the following proposition.
\end{proof}

\begin{prop}
\label{nomonodromy}
Let $\mathbf v,\mathbf w$ be as above. The monondromy automorphism $\sigma_\beta$ for every codimension two stratum in $X(\mathbf v,\mathbf w)$ is trivial.
\end{prop}
\begin{proof}

Let $Z(\mathbf v,\mathbf w) = Y(\mathbf v,\mathbf w)\times_{X(\mathbf v,\mathbf w)}Y(\mathbf v,\mathbf w)$, and let $\tilde{\pi}\colon Z(\mathbf v,\mathbf w) \to X(\mathbf v,\mathbf w)$ be the canonical map (induced from the projective morphism $\pi_\mathbf v \colon Y(\mathbf v,\mathbf w) \to X(\mathbf v,\mathbf w)$). By \cite[Theorem 7.2]{N98} $Z(\mathbf v,\mathbf w)$ is a Lagrangian subvariety of $Y(\mathbf v,\mathbf w)\times Y(\mathbf v,\mathbf w)$. Let $[Z(\mathbf v,\mathbf w)]_{\alpha}$ denote the set of components of $Z(\mathbf v,\mathbf w)$ in the closure of $\tilde{\pi}^{-1}(X(\mathbf v-\alpha,\mathbf w)^{\text{reg}})$. 

For an irreducible representation $L(\lambda)$ of the semisimple Lie algebra associated to our quiver with highest weight $\lambda$, let $L(\lambda)_\mu$ denote its $\mu$-weight space. Let $\eta = \mathbf w-\mathbf v \in P^+$. It is shown in \cite[Theorem 10.15]{N98} that there is a natural isomorphism between the formal linear span of $[Z(\mathbf v,\mathbf w)]_{\alpha}$ (viewed as a subspace of the top cohomology of $Z(\mathbf v,\mathbf w)$) and the tensor product $L(\eta+\alpha)_{\eta}^*\otimes L(\eta+\alpha)_{\eta}$. But now the number of components $s$ in $[Z(\mathbf v,\mathbf w)]_{\alpha}$ is just the the number of orbits of the monodromy action on the components of $\tilde{\pi}^{-1}(x)$ for $x \in X(\eta,\mathbf w)^{\text{reg}}$, of which there are clearly $r^2$, where $r$ is the number of components in $\pi_\mathbf v^{-1}(x)$. Now Theorem $10.2$ of \cite{N98} shows that the dimension of $L(\eta+\alpha)_{\eta}$ is exactly $r$, and hence using the natural isomorphism above, we find $s = r^2$, and hence the monodromy action must be trivial as required.
\end{proof}

\begin{example}
Suppose our quiver is $A_3$, and $\mathbf w=\Lambda_1+\Lambda_2+\Lambda_3$, $\mathbf v = \alpha_1+\alpha_2 + \alpha_3$. Then $X(\mathbf v,\mathbf w)$ is $4$ dimensional, and has four strata, labelled by $0,\alpha_1+\alpha_2,\alpha_2+\alpha_3$ and $\mathbf v$ itself, where the stratum corresponding to $0$ is a point, and each of $X(\alpha_1+\alpha_2,\mathbf w)^{\text{reg}}$ and $X(\alpha_2+\alpha_3,\mathbf w)^{\text{reg}}$ is two-dimensional. The slice through the stratum $X(\alpha_1+\alpha_2,\mathbf w)$ is isomorphic to $X(\alpha_3,2\Lambda_3)$ while that through $X(\alpha_2+\alpha_3,\mathbf w)$ is isomorphic to $X(\alpha_1,2\Lambda_1)$. These are both Kleinian singularities of type $A_1$, and so the symplectic Galois group in this case is $S_2\times S_2$, which is the parabolic subgroup of $S_4$ stabilizing the weight $\mathbf w - \mathbf v =\Lambda_2$.
\end{example}

Our calculation of the symplectic Galois group for finite type quiver varieties, along with the general theory of Section \ref{Basics}, have the following immediate consequence:

\begin{cor}
Let $\pi \colon Y(\mathbf v,\mathbf w) \to X(\mathbf v,\mathbf w)$ be as above, and let $\lambda = \Lambda_w - \alpha_v$. If $x \in X(\mathbf v,\mathbf w)$, then the cohomologies $H^*(\pi^{-1}(x))$ carry an action of the group $W_\lambda$.
\end{cor}

\begin{example}\label{affine example}
We give an example of a symplectic Galois group for a quiver variety of affine type. Let $Q_\Gamma$ be the affine quiver associated by the McKay correspondence to the finite subgroup $\Gamma$ of $\text{SL}_2(\C)$. Then if we take $\Qv = n\delta$ where $\delta$ is a generator for the imaginary roots, and $\Lw = \Lambda_0$, where $0$ denotes the extending node of the affine quiver, then it is known that $X(\mathbf v,\mathbf w)$ is isomorphic to 
$\C^{2n}/\Gamma_n$ where $\Gamma_n = \Gamma \wr S_n$.  Now for the quotient of a symplectic vector space by a symplectic reflection group such as $\Gamma_n$ it is known that the stratification by stabilizer type coincides with that of the symplectic leaves. It follows that there are two codimension two strata: one labelled by the group $\Gamma$ and the other by a subgroup of order two generated by a simple reflection in $S_n$ (of course if $n=1$ then this stratum does not exist). It can be checked that the symplectic Galois group is therefore either $W(\Gamma)$, the Weyl group associated to $\Gamma$, if $n=1$, or $S_2\times W(\Gamma)$ for $n>1$. Note that the action of the affine Weyl group $W_\text{aff}$ fixes the imaginary root $\delta$, and thus $\Lambda_0-\delta$ and $\Lambda_0 - n\delta$ for $n>1$ have the same stabilizer in $W_{\text{aff}}$, thus the naive generalization of Theorem \ref{QuiverWeylgroup} fails to hold. 
\end{example}

\begin{remark}
Many authors, including Nakajima \cite{N03}, Lusztig \cite{Lu} and Maffei \cite{Maffei}, have studied actions of the Weyl group on the cohomology of quiver varieties. It would be interesting to compare the actions constructed here with these actions. A novelty of our approach is that the action of the group $W_\lambda$ is constructed without the use of a presentation of $W_\lambda$: previous approaches define the action of simple reflections and then must explicitly check the braid relations to deduce the existence of a Weyl group action. On the other hand, these approaches give a more explicit description of the action of the simple reflections. 

A first step in such a comparison might be to exhibit an explicit homomorphism from the Weyl group to the symplectic Galois group. As the previous example shows, we should not expect this to be surjective in general, but it seems reasonable to expect that there should be a natural embedding. Evidence for this is provided, for example, by the generic isomorphisms constructed in \cite{Maffei}. We hope to address the existence of such an embedding in future work.
\end{remark}

\section{Bionic symplectic varieties and Springer theory}

In this section we consider a special class of conic symplectic resolutions, which we call \textit{bionic}. These are symplectic resolutions $\pi\colon Y \to X$ which, in addition to having a conical action $\lambda\colon\mathbb G_m \to \text{Aut}(Y)$ acting with positive weight on the symplectic form, have another action $\mathbb G_m$-action $\rho\colon \mathbb G_m \to \text{Aut}(Y)$, commuting with the $\lambda$-action which is required to be Hamiltonian (and thus preserves the symplectic form on $Y$). This class of resolutions has been considered by a number of authors, see for example \cite{Lo}, \cite{BLPW}.

\begin{lemma}
The actions $\lambda$ and $\rho$ naturally extend to an action on the deformation $\pi\colon \mathcal Y \to \mathcal X$. Moreover the Hamiltonian action acts trivially on the base of the deformation of $Y$, thus it induces an action on each deformation $Y_t$ ($t\in H^2(Y,\mathbb C)$).
\end{lemma}
\begin{proof}
The fact that the action $\lambda$ extends is proved in \cite{Na1}. The arguments in \cite[Lemma 20]{Naflips} show that the Hamiltonian action naturally extends to a Hamiltonian action on  the formal deformation. That argument moreover shows that the  action preserves the base of the deformation. Since $\rho$ commutes with $\lambda$, the action on the formal deformation extends to the algebraicization, and moreover again fixes the base of the deformation, and hence induces a Hamiltonian action on each deformation. 
\end{proof}

Though we will focus on the case where $Y^\rho$ is finite, it is checked in \cite{Lo} that the connected components of the fixed point locus are always themselves conical symplectic resolutions. 
 
\begin{prop}
\label{restriction varieties}
Each connected component of the fixed point locus $Y^{\rho}$ for the $\rho$-action on $Y$ is a conical symplectic resolution.
\end{prop}

\begin{lemma}
\label{finite fixed points etale}
Suppose that $\rho$ has finitely many fixed points on $Y$. Then $\mathcal Y^\rho$ is a trivial \'etale covering of $B_Y = H^2(Y,\mathbb C)$.
\end{lemma}
\begin{proof}
The map $q_Y\colon \mathcal Y \to B_Y$ is smooth, hence it is surjective on tangent spaces. The action of $\rho$ is trivial on $B_Y$, and if $x \in Y^\rho$ then $dq_Y$ must induce an isomorphism on the $\rho$-fixed subspace of $T_x\mathcal Y$ (it must be surjective, and the kernel of $dq_Y$ is the subspace $T_xY$, so since $x$ is an isolated fixed point of the $\rho$-action on $Y$, it must also be injective). The dimension of the $\rho$-fixed subspace in the tangent space of $\mathcal Y^\rho$ is locally constant, and hence because the contracting $\mathbb G_m$ action contracts $\mathcal Y$ to a smooth subvariety of $Y$, it must be the case that $\mathcal Y^\rho$ has $dq_{Y|\mathcal Y^{\rho}}$ an isomorphism everywhere. Thus $q_Y$ restricts to an \'etale map on $\mathcal Y^\rho$. 

By taking an equivariant compactification, one sees by the same argument that the fixed locus on the compactified locus is proper and \'etale, hence finite. If the restriction to the original locus was not finite, it can only be because the fibre dimension is non-constant. But it is known (see the next Lemma for more details) that the map $\mathcal Y \to H^2(Y,\C)$ is a smooth $C^\infty$-trivial fibration, and since the number of fixed points in a fibre $\mathcal Y_t$ is equal to the Euler characteristic of $\mathcal Y_t$, it follows that the fibre multiplicities are constant, and hence $q_Y$ is finite as required. Since $H^2(Y,\C)$ is simply-connected, it follows immediately that it is a trivial covering.
\end{proof}

Let $\rho_c\colon U(1)\to \text{Aut}(Y)$ be the restriction of the $\mathbb G_m$ action to the circle $U(1)$. 

\begin{lemma}
The map $\mathcal Y \to H^2(Y,\C)$ is a smooth $\rho_c$-equivariantly trivial fibration (where $U(1)$ acts trivially on the base).
\end{lemma}
\begin{proof}
The fact that the deformation $\mathcal Y\to H^2(Y,\C)$ is topologically trivial is noted in \cite{Na2}. One can prove it using the isotopy theorems of Thom. Moreover, the second isotopy theorem ensures that the trivialization can be made equivariant for the action of a compact group $K$. 

Note that in particular this implies, since the fixed points of the $\mathbb G_m$ action coincide with those of the $U(1)$ action, that the fixed point locus  $Y^{\rho}$ is a trivial fibration. 
\end{proof}

\begin{example}
Let $Y = T^*\mathcal B$ be the flag variety for a reductive group $G$ and $X = \mathcal N$ the nilpotent cone. It is well known that the moment map $\mu\colon Y \to X$ is a symplectic resolution.  If $T$ is a maximal torus of $G$ and $\rho\colon \mathbb G_m \colon T$ is a one-parameter subgroup, then the induced action of $T$ on $Y$ is Hamiltonian, so $\rho$ gives a Hamiltonian $\mathbb G_m$-action on $Y$. Let $L = G^\rho$ be the subgroup of $G$ fixed by the action of $\rho$. It is a Levi subgroup of $G$. The fixed point locus of the $\rho$-action is a disjoint union of varieties isomorphic to $T^*\mathcal B_L$ where $\mathcal B_L$ is the flag variety of $L$. The connected components are in bijection with $W/W_L$ where $W$ is the Weyl group of $G$ and $W_L$ is the Weyl group of $L$.  In this case the restriction of the moment map is in fact a symplectic resolution.
\end{example}

We now restrict again to the case where $\rho$ has finitely many fixed points. Note that in this case we have a permutation action of $W$ on the fixed point loci $Y_t^{\rho}$ due to the trivialization given by Lemma \ref{finite fixed points etale}. Let us denote by $P$ the corresponding permutation representation of $W$. We claim that the representation of the symplectic Galois group on $H^*(\pi^{-1}(0))$ is isomorphic to $P$. The strategy of the proof is to use localization theory in the equivariant cohomology of the $\rho$-action and the nearby cycles description of the Springer sheaf. The nearby cycles functor readily lifts to the equivariant derived category, and in the equivariant cohomology of the fixed points, the family $\mathcal Y^{\rho}\to H^2(Y,\C)$ is trivial, with the generic action of $W$ being given by the permutation action.

Fix a Hermitian metric on $\mathcal X$ which is invariant under the action of $U(1)$ the maximal compact subgroup of $\rho$, the Hamiltonian $\mathbb G_m$-action (since $X$ is affine this can be done globally), and a $W$-invariant metric on $H^2(Y,\C)$, so that it descends to give a metric on $B_X = H^2(Y,\C)/W$. The \textit{Milnor fibre} of the map $g\colon \mathcal X \to B_X$ at a point $x \in X$ is $F_{x,\epsilon} = B(x,\epsilon)\cap g^{-1}(v)$ where $d(0,v)\ll\epsilon \ll1$.

\begin{lemma}
The stalk of the Springer sheaf, as a $U(1)$-equivariant sheaf, at $x \in X$ is $H^*_{U(1)}(F_{x,\epsilon})$.
\end{lemma}
\begin{proof}
By Theorem \ref{nearby cycles2}, the Springer sheaf is given by a nearby cycles construction, and hence the Lemma follows directly from the construction of the nearby cycles functor. Note that as we have chosen our metric to be $U(1)$-invariant, so that $F_{L,\epsilon}$ is indeed a $U(1)$-space.
\end{proof}

\begin{theorem}
Suppose that the Hamiltonian $\mathbb G_m$ acts with finitely many fixed points. Then the symplectic Galois group $W$ permutes the fixed points in a generic deformation $\mathcal X_t$ and the cohomology $H^*(X)$ has a filtration whose associated graded is isomorphic as a $W$-representation to the permutation representation given by this action.
\end{theorem}
\begin{proof}
If the Hamiltonian $\mathbb G_m$-action $\rho$ has finitely many fixed points on $Y$, then we have seen that fixed point locus $\mathcal Y^{\rho}$ is a trivial finite cover of $H^2(Y,\mathbb C)$. Moreover, since the action of $\rho$ is compatible with the deformation it follows that the symplectic Galois group action commutes with the action of $\rho$ and so acts on $\mathcal Y^\rho$, hence the triviality of the covering $\mathcal Y^\rho\to H^2(Y,\C)$ yields a permutation action of $W$ on the sheets of the covering. 

Now the action of the symplectic Galois group on the stalks of the Springer sheaf is, via the nearby cycles description, given by the monodromy action coming from the local triviality of the Milnor fibre fibration over the regular locus in $B_X$. To prove the theorem we use localization theory for the $U(1)$-action, following the strategy of \cite{Tr}. Indeed we have a diagram:

\xymatrix{
& & & & F_{x,\epsilon}^{\rho} \ar[d] \ar[r] & F_{x,\epsilon} \ar[d] \\
& & & & H^2(Y,\C)^{\text{reg}} \ar[r] & B_X^{\text{reg}}
}

\noindent
where the lower arrow is the quotient map. This diagram induces a corresponding diagram where we replace the upper row with the corresponding Borel constructions for the $U(1)$-action (denoting by $X_{hU(1)}$ the Borel space associated to $X$):

\xymatrix{
& & & & (F_{x,\epsilon}^{\rho})_{hU(1)} \ar[d]^\psi \ar[r] & (F_{x,\epsilon})_{hU(1)} \ar[d]^\phi \\
& & & & H^2(Y,\C)^{\text{reg}} \ar[r]^\iota & B_X^{\text{reg}}
}

\noindent
The proof then proceeds as in \cite{Tr}: If we let $\mathcal F_\rho = \bigoplus_{i \in \mathbb Z} R^i\psi_*(\mathbb Q)$ and $\mathcal F = \bigoplus_{i \in \mathbb Z} R^i\phi_*(\mathbb Q)$, then each are locally constant sheaves with fibres isomorphic to the $U(1)$-equivariant cohomology of $F_{x,\epsilon}^{\rho}$ and $F_{x,\epsilon}$ respectively. Fixing an isomorphism of the rational cohomology $H^*(BU(1),\mathbb Q) \cong \mathbb Q[t]$, each of $\mathcal F$ and $\mathcal F_\rho$ are sheaves of $\mathbb Q[t]$-modules. It now follows from the localization theorem for $U(1)$-equivariant cohomology that we have an isomorphism:
\[
\mathcal F\otimes_{\mathbb Q[t]} \mathbb Q[t,t^{-1}] \cong \iota_*(\mathcal F_\rho)\otimes_{\mathbb Q[t]} \mathbb Q[t,t^{-1}]
\]
Now taking stalks, which we view as $W$-modules, this becomes the statement
\[
H^*_{U(1)}(F_{x,\epsilon},\mathbb Q)\otimes_{\mathbb Q[t]}\mathbb Q[t,t^{-1}] \cong  H^*(F_{x,\epsilon}^\rho,\mathbb Q)\otimes \mathbb Q[t,t^{-1}] \cong P\otimes  \mathbb Q[t,t^{-1}]
\]
(where on the left-hand side we have used the fact that $\rho$ acts trivially on $F_{x,\epsilon}^\rho$). Consider the right-hand side of this isomorphism.  We have a Leray spectral sequence in the category of locally constant sheaves associated to the composition:
\[
(F_{x,\epsilon})_{hU(1)} \to BU(1)\times B_X^{\text{reg}} \to B_X^{\text{reg}}.
\]
The stalks of the sheaves on the $E_2$-page are 
\[
E_2^{ij} = H^i(BU(1),\mathbb Q) \otimes H^j(F_{x,\epsilon},\mathbb Q).
\]
However, it is known that by \cite[Theorem 2.12]{Ka} the cohomology of the fibres of a symplectic resolution $\pi\colon Y\to X$ are nonzero only in even degrees. This ensures that the cohomology of $F_{x,\epsilon}$ vanishes in odd degrees and so the spectral sequence collapses at the $E_2$-page. 

Thus it follows that 
\[
H_{U(1)}^*(F_{x,\epsilon},\mathbb Q)\otimes_{\mathbb Q[t]} \mathbb Q[t,t^{-1}] \cong P\otimes  \mathbb Q[t,t^{-1}]
\]
admits a filtration whose associated graded is $E^{ij}_2\otimes \mathbb Q[t,t^{-1}]$. The conclusion of the theorem then follows by taking the $i+j=0$ line of the $E_2$-page.
\end{proof}

\section{Slices in the affine Grassmannian}

The idea that slices in the affine Grassmannian should have a kind of Springer theory associated to them goes back to work of Ngo, \cite[\S 2]{Ngo}, and was studied in the context of knot invariants by Kamnitzer in \cite{K}. In particular Section $2$ of \cite{K} contains many of the results of this section. Our main contribution is to compare the Beilinson-Drinfel'd deformation with the universal conic deformation and put the work of \cite{K} into our general context. 

\subsection{Background on Slices}

In this section we review some facts about slices in the affine Grassmannian. For more details see \cite{KWWY} and \cite{K}.

Let $G$ be a reductive algebraic group over $\mathbb C$. We write $G \llsb t^{\pm 1} \rrsb$ or $G \llb t^{\pm 1} \rrb$ for the $\mathbb C\llsb t^{\pm 1}\rrsb$ or $\mathbb C\llb t^{\pm 1}\rrb$ points of $G$, and similarly we write $G[t]$ for the $\mathbb C[t]$-points of $G$. The \textit{thick affine Grassmannian} is the scheme $\Gr = G\llb t^{-1} \rrb/G[t]$, which contains the thin affine Grassmannian via the maps
\[
G\llb t \rrb /G[\![t]\!] \cong G[t, t^{-1}]/G[t] \hookrightarrow G \llb t^{-1} \rrb/G[t].
\]

\noindent
A coweight $\lambda$ for $G$ may be viewed as an element of $G[t,t^{-1}]$ and hence yields a point in the thin affine Grassmannian, which we denote by $t^\lambda$. Set, for $\lambda$ and $\mu$ dominant coweights of $G$,
\[
\Gr^{\lambda} = G[t].t^\lambda, \quad \Gr_{\mu} = G_1[[t^{-1}]].t^{w_0\mu}
\] 
where $G_1\llsb t^{-1}\rrsb$ is the kernel of the natural homomorphism $G\llsb t^{-1}\rrsb \to G$ given by evaluation at $t^{-1}=0$ and as usual $w_0$ denotes the longest element of the Weyl group of $G$. 

For $\lambda$ and $\mu$ as above with $\mu\leq \lambda$, let $\Gr_{\mu}^{\bar{\lambda}} =\overline{\Gr^{\lambda}} \cap \Gr_{\mu}$. This is an affine slice to the orbit $\Gr^{\mu}$ through the point $t^{w_0\mu}$, since $\Gr_\mu$ is transverse to every $\Gr^{\nu}$ and intersects $\Gr^{\mu}$ precisely at the point $t^{w_0\mu}$. It is known to have dimension $2 \langle \rho, \lambda-\mu\rangle$ (where $\rho$ as usual denotes the half-sum of positive roots).
There is a natural action of $\mathbb G_m$ on $\Gr$ by ``loop rotation''. This action preserves both $G[t]$- and $G\llb t^{-1} \rrb$-orbits, and hence preserves each $\Gr^{\lambda}_\mu$.

\begin{lemma}
Let $\lambda$ and $\mu$ be dominant coweights of $G$ such that $\mu\leq \lambda$. Then the variety $\Glm$ is a conical symplectic variety in the sense of Kaledin \cite{Ka}. Moreover, its symplectic leaves are $\Gr^{\lambda}_\mu = \Gr^\lambda\cap \Gr_\mu$.
\end{lemma}
\begin{proof}
This is contained in Theorems 2.5 and 2.7 of \cite{KWWY}. The existence of the Poisson structure is a consequence of the the fact that $(\gf\llb t^{-1} \rrb,\gf[t], \gf_1\llsb t^{-1}\rrsb)$ is a Manin triple (where of course $\gf_1\llsb t^{-1}\rrsb = t^{-1}\gf\llsb t^{-1}\rrsb$) with the pairing between $\gf[t]$ and $t\gf\llsb t^{-1} \rrsb$ being given by the residue of the Killing form. The loop rotation action preserves $\Glm$ and contracts it to the unique fixed point $t^{w_0\mu}$, and its compatibility with the Manin triple described above ensures it makes $\Glm$ into a conic symplectic variety as claimed (acting with weight $-1$ on the Poisson bracket).
\end{proof}

The symplectic singularities $\Glm$ do not always possess a symplectic resolution, but there is always a $\mathbb Q$-factorial terminalization which can be described explicitly, which is therefore a resolution in the cases where such resolutions exist. Indeed given $\lambda$, assuming (as we shall) that $G$ is of adjoint type, we may pick a sequence $\vec{\lambda}= (\lambda_1,\ldots,\lambda_n)$ of fundamental coweights such that $\lambda = \lambda_1+ \ldots \lambda_n$. 

Recall that the convolution product of subsets $A$ and $B$ in $\Gr$ is defined to be $A\tilde{\times} B = p^{-1}(A)\times_{G[t]}B$ where $p\colon G\llb t^{-1}\rrb\to \Gr$ is the natural projection map. The variety $\overline{\Gr^{\vec{\lambda}}}$ is then defined to be $\overline{\Gr^{\lambda_1}}\tilde{\times} \ldots \tilde{\times} \overline{\Gr^{\lambda_n}}$, which is naturally equipped with a projection map $m \colon \overline{\Gr^{\vec{\lambda}}}\to \overline{\Gr^{\lambda}}$. If we set $\Gr^{\bar{\vec{\lambda}}}_\mu = m^{-1}(\Gr_\mu)$ then $\Gr^{\bar{\vec{\lambda}}}_\mu$ is a $\mathbb Q$-factorial terminalization of $\Gr_\mu^{\lambda}$, and therefore, by results of Namikawa, a symplectic resolution of $\Gr^{\bar{\lambda}}_\mu$ whenever such a resolution exists. Provided the fundamental coweights $\lambda_i$ are minuscule as weights for the dual group $G^\vee$, the orbits $\Gr^{\lambda_i}$ are actually closed, and hence the convolution product $\overline{\Gr^{\vec{\lambda}}}$ is smooth. A refinement of this observation shows the following:

\begin{theorem}
\label{resolutions for slices}\cite[Theorem 2.9]{KWWY}.
The variety $\Glm$ has a symplectic resolution precisely when there do not exist coweights $\nu_1,\ldots,\nu_k$ such that $\nu_1+\ldots+\nu_n = \mu$ and for all $k$, $\nu_k$ is a weight of $V(\lambda_k)$ (the irreducible representation of highest weight $\lambda_k$ for the dual group $G^\vee $) and for some $k$, $\nu_k$ is not an extremal weight of $V(\lambda_k)$.
\end{theorem}

\begin{remark}
We shall assume from now on that our variety $\Glm$ possesses a symplectic resolution. In this case, one can check that the preimage of $\Glm$ lies in the open convolution space $\Gr^{\lambda_1}\tilde{\times} \ldots \tilde{\times}\Gr^{\lambda_n}$, hence we will write $\Glmr$ for the resolution.
\end{remark}

\begin{remark}
The theorem shows that we must take $\lambda$ to be a sum of miniscule coweights, if we wish to ensure $\Glm$ always has a symplectic resolution. When $G$ is of type $A$, every fundamental weight is minuscule, hence this is automatic, but in general we see that the theorem gives a large supply of symplectic varieties which do not possess a symplectic resolution.
\end{remark}

Namikawa's recipe for computing the symplectic Galois group can be carried out in these cases thanks to the work of \cite[Theorem 5.2]{MOV}: 

\begin{lemma}
The codimension two strata of the symplectic variety $\Glm$ are those $\Gr^{\lambda}_\nu$ with $\nu$ dominant and $\nu = \lambda - \check{\alpha}_i\geq \mu$, for some fundamental coroot $\check{\alpha}_i$. Moreover, a transverse slice through this stratum is isomorphic to a Kleinian singularity of the form $A_{p_i-1}$ where $p_i=\langle \lambda_i,\lambda\rangle$.
\end{lemma}
\begin{proof}
This follows immediately from Theorem 5.2 of \cite{MOV} combined with the dimension formula $\dim(\Glm)=2\langle \rho,\lambda-\mu\rangle$.
\end{proof}

Once one has computed the codimensional $2$ strata, in order to calculate the symplectic Galois group one must compute the monodromy action of the fundamental group of the strata on the exceptional fibres of minimal resolutions of the codimension two strata. However, it is asserted in \cite{KWWY} that these strata are simply-connected, and hence the monodromy is trivial.

\begin{prop}
\label{W for slices}
Let $Y = \Gr^{\lambda}_0$, and assume that the conditions of Theorem \ref{resolutions for slices} hold. Then $Y$ has a symplectic resolution and if $\lambda = \sum_{i=1}^r p_i\lambda_i$ then its symplectic Galois group is isomorphic to\footnote{By convention, we take $S_1$ to be the trivial group.} $\prod_{p_i>0} S_{p_i}$. 
\end{prop}
\begin{proof}
Theorem \ref{resolutions for slices} shows that $Y$ has a symplectic resolution, and we have seen that the slices through codimension two strata are Kleinian singularities of type $A_{p_i+1}$, and thus the associated reflection group is the symmetric group on $p_i+1$ letters. To be more precise, given $\lambda$ as in the statement of the Proposition, we obtain a codimension two stratum corresponding to $\nu = \lambda - \check{\alpha}_k$ provided $\nu$ is dominant. But if $C= (c_{ij})$ denotes the Cartan matrix of dual group $G^\vee$, then 
\[
\nu(\alpha_j) = p_j - c_{kj},
\]
which is at least $p_j$ if $j \neq k$, while if $j=k$ then $\nu(\alpha_k)=p_k-2$.  We must thus have $p_i>1$, however, if $p_i=1$ then by our convention the factor $S_{p_i}$ does not actually contribute to the product, hence the result follows.
\end{proof}

\begin{remark}
If we take $\mu\neq 0$ then the symplectic Galois group is in general smaller: one needs in addition, for each fundamental weight $\lambda_i$, that we have $\mu \leq \lambda-\alpha_i$.
\end{remark}

Next we recall the geometric Satake correspondence, discovered by Lusztig, whose full categorical meaning was elucidated in the beautiful paper \cite{MV}. This establishes an equivalence of Tannakian categories between the representations of the Langlands dual group $G^\vee$ and the category of $G\llsb t\rrsb$-equivariant perverse sheaves on $\Gr$. Under this equivalence, the convolution product of sheaves (recalled above for subsets of $\Gr$) corresponds to tensor product of representations. As such one sees that if $m\colon \Gr^{\bar{\vec{\lambda}}}_0\to \Gr^{\bar{\lambda}}_0$ is a symplectic resolution, then  the associated Springer sheaf $\mathsf{Spr}_{\lambda}$ corresponds to the tensor product $V(\lambda_1)\otimes\ldots\otimes V(\lambda_n)$ of representations of $G^\vee$. It follows that 

\[
\mathsf{Spr}_{\vec{\lambda}} \cong \bigoplus_{\nu} i^*(\text{IC}(\mathcal O_\nu))\otimes M_{\vec{\lambda},\nu},
\] 
where $M_{\vec{\lambda},\nu}$ is the multiplicity space, a vector space of dimension equal to the multiplicity of the irreducible $V_\nu$ in the tensor product representation $V_{\lambda_1}\otimes\ldots\otimes V_{\lambda_n}$ and $i$ is the inclusion of our slice into the affine Grassmannian. (Note that these multiplicities are computable, via a number of combinatorial tools, \textit{e.g.} crystal bases, Littelmann paths \textit{etc.}.)

\subsection{Comparing the Universal Deformation and the BD Deformation}
\label{universal BD}

The 
\linebreak 
Beilinson-Drinfeld deformation family arises by considering the affine Grassmannian spread over $\mathbb P^1$. We equip $\mathbb P^1$ with a weight $1$ $\mathbb G_m$-action, fixing points $\{0,\infty\} \subseteq \mathbb P^1$ and pick a local coordinate $t\in \Gamma(\mathbb P^1, \mathcal K_\mathbb P^1)$ at $0$ having a pole at $\infty \in \mathbb P^1$. The basic observation is the following: consider the set
\[
\begin{split}
\{(E,\phi): E & \text{ is a principal $G$-bundle on } \mathbb P^1 \text{ and } \phi\colon E_{|\mathbb P^1\backslash\{0\}} \to E^0_{|\mathbb P^1\backslash\{0\}} \\
& \text{ is an isomorphism}\},
\end{split}
\]
where $E^0$ denotes the trivial $G$-bundle. Using our local coordinate $t$ at $0$ one can readily obtain a bijection between this set and the affine Grassmannian $\Gr$.

This description naturally allows us to deform the affine Grassmannian over $\mathbb A^n$: we write $\Gr_{\mathbb A^n}$ for the family over $\mathbb A^n$ whose fibre at $(a_1,\ldots,a_n)$ is just
\[
\begin{split}
\Gr_{(a_1,\ldots,a_n)} = \{(E,\phi): &  E \text{ is a principal } G\text{-bundle on } \mathbb P^1 \\
&\text{ and }   \phi\colon E_{|\mathbb P^1\backslash\{a_1,\ldots,a_n\}} \to E^0_{|\mathbb P^1\backslash\{a_1,\ldots, a_n\}} \text{ is an isomorphism}\}   
\end{split}
\]
where as above $E^0$ is the trivial bundle. We write $\Gr_{\mu,\mathbb A^n}$ for the set of pairs $(E,\phi)$ such that $E$ has type $\mu$ at $\infty$. Note that this family in fact descends to one over $\text{Sym}^n(\mathbb A^1)$.

Define the convolution Grassmannian to be the $\mathbb A^n$-space $\tilde{\Gr}_{\mathbb A^n}$, where the fibre over $(a_1,\ldots,a_n)$ given by
\[
\begin{split}
\tilde{\Gr}_{(a_1,\ldots,a_n)} = &\{(E^1,E^2,\ldots,E^n,\phi_1,\ldots,\phi_n): E_i \text{ a $G$-bundle},\\ 
&\phi_i\colon E_{i|\mathbb P^1\backslash a_i} \to E_{i-1,\mathbb P^1\backslash \{a_i\}} \text{ an isomorphism}\}. 
\end{split}
\]

 There is a natural morphism from this space to $\Gr_{(a_1,\ldots,a_n)}$ given by sending $(E^1,\ldots,E^n,\phi_1,\ldots,\phi_n)$ to $(E^n, \phi_1\circ\ldots \circ \phi_n)$ (the composition being well-defined on $\mathbb A^1\backslash \{a_1,\ldots, a_n\}$).

Now fix a coweight $\lambda$ and consider the slice $\Gl_0$ through the orbit closure $\overline{\Gr^\lambda}$. Taking a decomposition $\vec{\lambda} = (\lambda_1,\ldots,\lambda_n)$ of $\lambda$ into a sum of minuscule weights, so that $\lambda = \lambda_1 + \ldots + \lambda_n$, the convolution space $Y=\Glr$ is a resolution of $\Gl_0$. Let $n_i$ be the multiplicity with which the fundamental weight $\varpi_i$ occurs in the sequence $(\lambda_1,\lambda_2,\ldots,\lambda_n)$ and let $\Sigma_{\vec{\lambda}} = \prod_{i=1}^\ell S_{n_i}$ (where $\ell$ is the rank of $G$). 

The space $\tilde{\Gr}_{\mathbb A^n}$ above yields a deformation of our symplectic resolutions $\Gr^{\vec{\lambda}}_{\mu}$ over $\mathbb A^n$, by imposing the condition that, at each $a_i$, that the morphism $\phi_i$ have Hecke-type $\lambda_i$. Imposing the condition that, for each $x \in (a_1,\ldots,a_n)$, the morphism $\phi$ has Hecke type at most $\sum_{i: a_i =x} \lambda_i$, our affine deformation no longer lives over $\text{Sym}^n\mathbb A^n$, rather it lives over $\mathbb A^n/\Sigma_{\vec{\lambda}}$. We thus obtain a diagram:

\xymatrix{
& & & & & \tilde{\Gr}_{\mathbb A^n}^{\vec{\lambda}}\ar[r] \ar[d] & \mathbb A^n \ar[d] \\
& & & & & \Gr_{\mathbb A^n}^{\lambda} \ar[r] & \mathbb A^n/\Sigma_{\vec{\lambda}}
}

It follows from the work of Namikawa recalled in $\S$\ref{sec:Conic universality} that there is a linear map $\alpha\colon \mathbb A^n \to H^2(Y,\mathbb C)$ such that the Beilinson-Drinfel'd deformation is obtained by pulling back the universal conic deformation via $\alpha$. Thus we have the following diagram:

\xymatrix{
& \tilde{\Gr}_{\mathbb A^n}^{\vec{\lambda}} \ar[rd] \ar[r]^a & \Gr_{\mathbb A^n}\times_{(f,q_{BD})} \mathbb A^n \ar[r]^{\tilde{f}} \ar[d]^{\tilde{q}_{BD}} & \mathbb A^n \ar[d]^{q_{BD}} \ar[r]^\alpha & H^2(\Gr^{\vec{\lambda}}_0) \ar[d]^{q} \\
& &  \Gr_{\mathbb A^n}^{\lambda}\ar[r]^f & \mathbb A^n/\Sigma_{\vec{\lambda}} \ar[r]^\beta & H^2(\Gr^{\vec{\lambda}}_0)/W,
}

\noindent
where $W$ is the symplectic Galois group, and the maps $\alpha$ and $\beta$ are the classifying maps given by (conic) universality. By considering the family $\Gr_{\mathbb A^n}\times_{(f,q_{BD})}\mathbb A^n$, and again using universality, we see that the right-hand square must commute.

Now let $\mathcal D$ be the discriminant locus of the universal conic deformation, so that $\mathcal D$ is a collection of hyperplanes in $H^2(Y,\mathbb C)$ invariant under the action of the symplectic Galois group $W$. The pull-back of this collection of hyperplanes under $\alpha$ is therefore the locus in $\mathbb A^n$ over which the deformations of $Y$ are not isomorphic to the corresponding Poisson deformations. But in the case of the BD Grassmannian, this is precisely the ``big $\vec{\lambda}$-diagonal''
\[
\Delta_{\vec{\lambda}} = \{(a_1,\ldots,a_n): a_i = a_j, \exists i \neq j \text{ but } \lambda_i = \lambda_j\}
\]

Note that in particular, we see that the image of $\alpha$ is not contained in $\mathcal D$.

Let $\mathbf B_{\vec{\lambda}} = \pi_1({\Delta_{\vec{\lambda}}}^c/\Sigma_{\vec{\lambda}})$, and let $\mathcal B_{\underline{\lambda}} = \pi_1(\mathcal D^{c}/W)$ be the fundamental group of complement of the image of $\mathcal D$ under the quotient map $q$ to $H^2(Y,\mathbb C)/W$. Now by Remark \ref{smallness of conic deformations}, the nearby cycles construction of the Springer sheaf $\Spr_{\vec{\lambda}}$ extends to this setting, and thus it follows that we obtain a monodromy action of $\mathcal B_{\vec{\lambda}}$ on $\Spr_{\vec{\lambda}}$.
 But by construction, the action is pulled-back from the action of $\mathbf B_{\underline{\lambda}}$ via the homomorphim $\alpha_*\colon \pi_1(\Delta_{\vec{\lambda}}^c/\Sigma_{\vec{\lambda}}) \to \pi_1(\mathcal D^c/W))$. Thus we have the following:

\medskip

\xymatrix{
& & & & \mathbf B_{\vec{\lambda}} \ar[r]^{\alpha_*} & \mathcal B_{\vec{\lambda}} 
\ar[d] \ar[dr] &
\\
& & & & & W \ar[r] & \text{Aut}(\Spr_{\vec{\lambda}})
}

To identify $\Sigma_{\vec{\lambda}}$ with the symplectic Galois group we will use the following Lemma, which was first observed in \cite[Proposition 2.8]{K}. We outline a proof for the reader's convenience.

\begin{lemma}
\label{Satake compatibility}
The monodromy action of $\Sigma_\lambda$ on $\Spr$ corresponds, under the geometric Satake correspondence, to the action by permutation of tensor factors.
\end{lemma}
\begin{proof}
This is essentially a consequence of the fusion construction of the commutativity constraint in \cite{MV}. In particular, it is shown in $\S 5$ of that paper that the convolution product can be interpreted via the fusion product, and in the proof of Lemma $6.1$ of $\S 6$ that the push-forward to $\mathbb A^1 \times \mathbb A^1$ of the fusion product sheaf has constant cohomologies. These facts imply the required compatibility.  
\end{proof}

The following Lemma is proved in type $A$ in \cite[\S 2]{K} (the proof in the general case is essentially the same).

\begin{lemma}
The monodromy action of $\mathbf B_{\vec{\lambda}}$ on $\Spr_{\vec{\lambda}}$ factors through the quotient $\Sigma_{\vec{\lambda}}$. 
\end{lemma}
\begin{proof}
Using the geometric Satake equivalence, the cohomology $H^*(\Glr)$ can be identified with the tensor product $V_{\lambda_1}\otimes\ldots \otimes V_{\lambda_n}$, where the $V_{\lambda_i}$ are the fundamental representations of $G^\vee$ the Langlands dual of $G$, and moreover by Lemma \ref{Satake compatibility} the monodromy action of $\mathbf B_{\vec{\lambda}}$ on this cohomology factors through $\Sigma_{\vec{\lambda}}$ and is identified with the natural permutation action of $\Sigma_{\vec{\lambda}}$ on the tensor product. Since this action is faithful and is induced by the monodromy action on $\Spr_{\vec{\lambda}}$ the result follows.
\end{proof}

Thus our map $\alpha_*$ induces a homomorphism $\gamma \colon \Sigma_{\vec{\lambda}} \to W$, which gives a faithful action of $\Sigma_{\vec{\lambda}}$ on the cohomology $H^*(\Glr)$. In particular, $\gamma$ must be injective. But then since $W$ and $\Sigma_{\vec{\lambda}}$ have the same order (by Proposition \ref{W for slices}) $\gamma$ is an isomorphism.

It follows that we can reduce the calculation the action of the symplectic Galois group on the Springer sheaf $\mathcal S_{\vec{\lambda}}$ to the question of calculating the action of $\prod_{i \in I} S_{n_i}$ on the multiplicity spaces $M_{\vec{\lambda},\nu}$ (in the notation of the preceeding subsection). We use this in the next section to show that the multiplicity spaces in the Springer sheaf need not always be irreducible.

\section{Tensor products of representations}

The irreducible representations of $\mathfrak{sl}_{n+1}$, as for any semisimple Lie algebra, are indexed by highest weights. In the case of $\mathfrak{sl}_{n+1}$ however, these can naturally be identified with partitions $\lambda = (\lambda_1,\lambda_2,\ldots,\lambda_k)$ where $k\leq n$. Indeed the root lattice of $\mathfrak{sl}_{n+1}$ is $Q= \{\alpha \in \mathbb Z^{n+1}: \sum_{i=1}^{n+1} \alpha_i=0\}$, and hence the weight lattice $P$ is naturally identified with $\mathbb Z^{n+1}/\mathbb Z.\mathbf 1$ where $\mathbf 1 = (1,1,\ldots,1) \in \mathbb Z^{n+1}$. The set of dominant weights $P^+$ consists of the elements $\lambda \in P$ whose representatives $\tilde{\lambda} = (\tilde{\lambda}_1,\ldots,\tilde{\lambda}_{n+1}) \in \mathbb Z^{n+1}$ satisfy $\tilde{\lambda}_i \geq \tilde{\lambda}_{i+1}$ for $1\leq i \leq n$. Clearly given $\lambda \in P$ it has a unique representative $\tilde{\lambda}$ with $\tilde{\lambda}_{n+1}=0$. Thus any dominant weight $\lambda$ of $\mathfrak{sl}_{n+1}$ has a partition with at most $n$ parts associated to it, and we will write this partition as $(\lambda_1,\ldots,\lambda_n)$. 

Write  $V(\lambda)$ for the irreducible highest weight representation with highest weight $\lambda$. Given $\lambda, \mu,\nu \in P^+$, the Littlewood-Richardson coefficient $c_{\mu,\nu}^{\lambda}$ is the multiplicity of $V(\lambda)$ in $V(\mu)\otimes V(\nu)$ and the Littlewood-Richardson rule gives a combinatorial method for determining the numbers $c^{\lambda}_{\mu,\nu}$.

In more detail, the coefficient $c_{\mu,\nu}^\lambda$ is the number of skew tableaux of shape $\lambda/\mu$ and weight $\nu$, where a weighting is a filling of the skew tableau with $\nu_1$ 1s, $\nu_2$ 2s, \textit{etc.} such that the rows of the skew tableau are weakly increasing, and the columns are strictly increasing (that is, a \textit{semistandard} skew-tableau), and such that the word obtained by concatenating the reverse words of each row is a lattice word, that is the multiplicity of $i$ is at least that of $i+1$ on any initial segment of the word.

The fundamental representations of $\mathfrak{sl}_{n+1}$ are the exterior powers of $V$, the vector representation. They are the representations associated to the partitions $(1^k)$ for $k=1,\ldots n$. They are minuscule representations (that is, their weights form a single Weyl-group orbit).

\begin{lemma}
If $\nu = (1^k)$ then $c_{\mu,\nu}^\lambda \in \{0,1\}$.
\end{lemma}
\begin{proof}
The skew-shape $\lambda/\mu$, if it is defined, must have only one box per row, and hence it is clear that to obtain a Littlewood-Richardson tableau one must fill the boxes with the numbers $1,\ldots,k$ in order working down the columns from right to left. Thus $c_{\mu,\nu}^\lambda \in \{0,1\}$ as required.
\end{proof}

Using the previous lemma, we obtain the following: 

\begin{lemma}
Let $\mathfrak g = \mathfrak{sl}_{n+1}$, where we require $n \geq 9$ and let $\lambda = \varpi_3$ be the third fundamental weight, so that $\lambda$ has partition $(1^3)$. Then if $\mu = (2^3,1^3)$ then the multiplicity space $\text{Hom}_{\mathfrak g}(V(\mu),V(\lambda)^{\otimes 3})$ is $4$-dimensional. 
\end{lemma}
\begin{proof}
Using the Littlewood-Richardson Rule and the previous lemma one sees that the dimension is the number of ways of removing three boxes from the Young diagram of shape $(2^3,1^3)$ in such a way that no two boxes belong to the same row, and the resulting shape is still a Young diagram. It is easy to check that this can be done in precisely $4$ ways -- one may take $k$ boxes from the first column and $3-k$ from the second for any $k$ between $0$ and $3$.
\end{proof}

Recall that we say a local system $\mathcal E$ on a symplectic leaf is a Springer local system if $IC(\mathcal E)$ occurs as a constituent of the Springer sheaf.

\begin{cor}
Let $\lambda = 3\varpi_3$ and $\mu=0$, and let $\vec{\lambda} = (\varpi_3,\varpi_3,\varpi_3)$. We may decompose the Springer sheaf $\mathcal S$ for $\pi\colon \Glmr \to \Glm$ as
\[
\mathcal S \cong \bigoplus_{\nu, \nu\leq \lambda} IC(\Gr^{\mu}_0)\otimes \mathcal M_\mu,
\]
where each $\mathcal M_\mu$ carries a natural action of $W= S_3$ the symplectic Galois group. When $\mu = (2^3,1^3)$ this representation is reducible. Thus the symplectic Springer correspondence does not necessarily yield a function from Springer local systems to irreducible representations of the symplectic Galois group.
\end{cor}
\begin{proof}
The Geometric Satake Correspondence shows that global sections functor can be enhanced to yield an equivalence of categories from the category of perverse sheaves on $\Gr$ which are constructible with respect to the $G\llsb t\rrsb$-orbit stratification to the category of representations for the Langlands dual group $G^\vee$ where convolution corresponds to tensor product of representations. It follows that the multiplicity spaces $\mathcal M_\mu$ in the decomposition of $\mathcal S$ for our choice of variety $\Glmr$ are simply the tensor product multiplicities for the irreducible $V(\mu)$ in $V(\lambda)^{\otimes 3}$. When $\mu=(2^3,1^3)$ we have seen that this multiplicity is $4$, hence $\mathcal M_\mu$ is $4$-dimensional, and hence cannot be an irreducible representation of $S_3$. 
\end{proof}

\end{document}